\documentclass[11pt]{amsart}
\usepackage{amsfonts,latexsym,amsthm,amssymb,graphicx}
\usepackage[all]{xy}
\usepackage{hyperref}
\usepackage[usenames]{color} 

\DeclareFontFamily{OT1}{rsfs}{}
\DeclareFontShape{OT1}{rsfs}{n}{it}{<-> rsfs10}{}
\DeclareMathAlphabet{\mathscr}{OT1}{rsfs}{n}{it}

\CompileMatrices

\newtheorem{theorem}{Theorem}[section]
\newtheorem{prop}[theorem]{Proposition}
\newtheorem{lemma}[theorem]{Lemma}
\newtheorem{corol}[theorem]{Corollary}
\newtheorem{claim}[theorem]{Claim}
\newtheorem{conj}{Conjecture}

{\theoremstyle{definition} }
{\theoremstyle{remark} \newtheorem{remark}[theorem]{Remark}
\newtheorem{example}[theorem]{Example}}

\numberwithin{equation}{section}

\newcommand{\Abb}{{\mathbb{A}}}
\newcommand{\Lbb}{{\mathbb{L}}}
\newcommand{\Qbb}{{\mathbb{Q}}}
\newcommand{\hM}{{\widehat{M}}}
\newcommand{\oM}{{\overline{M}}}
\newcommand{\oP}{{\overline{P}}}
\newcommand{\hchi}{{\widehat{\chi}}}
\newcommand{\one}{1\hskip-3.5pt1}
\newcommand{\saf}{\,}
\newcommand{\caM}{{\mathcal{M}}}
\newcommand{\ocaM}{{\overline\caM}}
\newcommand{\fs}{{\mathfrak s}}
\newcommand{\fS}{{\mathfrak S}}
\DeclareMathOperator{\tr}{tr}
\DeclareMathOperator{\rk}{rk}
\DeclareMathOperator{\St}{St}
\DeclareMathOperator{\Sh}{Sh}
\newcommand{\qede}{\hfill $\lrcorner$}

\title{Explicit formulas for the Grothendieck class of $\ocaM_{0,n}$}
\author{Paolo Aluffi}
\author{Matilde Marcolli}
\author{Eduardo Nascimento}
\address{
Mathematics Department, 
Florida State University,
Tallahassee FL 32306, U.S.A.
}
\address{
Department of Mathematics, 
California Institute of Technology,
Pasadena CA 91105, U.S.A.
}
\email{aluffi@math.fsu.edu}
\email{matilde@caltech.edu}
\email{ehnascim@caltech.edu}

\begin{document}

\begin{abstract}
We obtain explicit expressions for the class in the Grothendieck group of varieties
of the moduli space $\ocaM_{0,n}$ of genus~$0$ stable curves with $n$ marked points. 
This information is equivalent to the Poincar\'e polynomial; it implies explicit expressions 
for the Betti numbers of the moduli space in terms of Stirling numbers or, alternatively, 
Bernoulli numbers.

The expressions are obtained by solving a differential equation characterizing the generating
function for the Grothendieck class as shown in work of Yuri Manin from the 
1990s. This differential equation is equivalent to S.~Keel's recursion for the Betti numbers
of $\ocaM_{0,n}$. Our proof reduces the solution to two combinatorial identities which follow 
from applications of Lagrange series.

We also study generating functions for the individual Betti numbers. In previous work it had
been shown that these functions are determined by a set of polynomials $p^{(k)}_m(z)$,
$k\ge m$, with positive rational coefficients, which are conjecturally log-concave. We 
verify this conjecture for many infinite families of polynomials $p^{(k)}_m(z)$, corresponding 
to the generating functions for the $2k$-Betti numbers of $\ocaM_{0,n}$ for all $k\le 100$. 
Further, studying the polynomials $p^{(k)}_m(z)$ allows us to prove that the generating 
function for the Grothendieck class of $\ocaM_{0,n}$ may be written as a series of rational 
functions in $\Lbb$ and the principal branch of the {\em Lambert W-function.\/}

We include an interpretation of the main result in terms of Stirling matrices and a discussion
of the Euler characteristic of $\ocaM_{0,n}$.
\end{abstract}

\maketitle

{\let\thefootnote\relax\footnotetext{2024/07/26}}

%%%

\section{Introduction}\label{sec:intro}

Let $\ocaM_{0,n}$ be the moduli space of stable $n$-pointed curves of genus $0$, 
$n\ge 3$, and denote by $[\ocaM_{0,n}]$ its class in the Grothendieck ring of varieties. 
This class is a polynomial in the Lefschetz-Tate class $\Lbb=[\Abb^1]$, and in fact
\[
[\ocaM_{0,n}] = \sum_{k=0}^{n-3} \rk H^{2k} (\ocaM_{0,n})\, \Lbb^k\saf,
\]
see \cite[Remark~3.2.2]{MR3701904}. Thus, this class is a manifestation of the Poincar\'e
polynomial of $\ocaM_{0,n}$. The following surprisingly elegant formula is, to our knowledge,
new.

\begin{theorem}\label{thm:M0nbarGC}
\begin{equation}\label{eq:M0nbarGC}
[\ocaM_{0,n}]=(1-\Lbb)^{n-1}
\sum_{k\ge 0}\sum_{j\ge 0}
s(k+n-1,k+n-1-j)\, S(k+n-1-j,k+1)\, \Lbb^{k+j}\saf.
\end{equation}
\end{theorem}

Here $s$, resp., $S$ denote Stirling numbers of the first, resp., second kind. The result implies
that the stated expression evaluates to a degree-$(n-3)$ polynomial in $\Lbb$ with positive 
coefficients, a fact that seems in itself nontrivial.
 
An explicit formula for the individual Betti numbers is an immediate consequence
of Theorem~\ref{thm:M0nbarGC}.

\begin{corol}\label{cor:bettisim}
For $n\ge 3$:
{\small
\[
\dim H^{2\ell}(\ocaM_{0,n})
=\sum_{j=0}^\ell \sum_{k=0}^{\ell-j} (-1)^{\ell-j-k}\binom{n-1}{\ell-j-k}
s(k+n-1,k+n-1-j)S(k+n-1-j,k+1)\saf.
\]}%
\end{corol}

For instance,
\begin{align*}
\dim H^6 &(\ocaM_{0,5}) =
s(4, 1) S(1, 1) - 4 s(4, 2) S(2, 1) + 6 s(4, 3) S(3, 1) - 4 s(4, 4) S(4, 1) \\
&\qquad\qquad + s(5, 3) S(3, 2)  - 4 s(5, 4) S(4, 2) +  6 s(5, 5) S(5, 2) \\
&\qquad\qquad + s(6, 5) S(5, 3) - 4 s(6, 6) S(6, 3) + s(7, 7) S(7, 4) \\
&=(-6)\cdot 1 - 4\cdot 11\cdot 1 +6\cdot(-6)\cdot 1 - 4\cdot 1\cdot 1
+35\cdot 3 \\
&\qquad\qquad -4\cdot (-10)\cdot 7 +6\cdot 1\cdot 15 
+ (-15)\cdot 25 - 4\cdot 1\cdot 90 + 1\cdot 350 \\
&=0
\intertext{while}
\dim H^6 &(\ocaM_{0,10}) =
s(9,6) S(6, 1) - 9 s(9, 7) S(7, 1) + 36 s(9,8) S(8, 1) - 84 s(9,9) S(9, 1) \\
&\qquad\qquad + s(10, 8) S(8, 2)  - 9 s(10,9) S(9, 2) +  36 s(10,10) S(10, 2) \\
&\qquad\qquad + s(11, 10) S(10, 3) - 9 s(11, 11) S(11, 3) + s(12,12) S(12, 4) \\
&=(-4536)\cdot 1 - 9\cdot 546\cdot 1 +36\cdot(-36)\cdot 1 - 84\cdot 1\cdot 1
+870\cdot 127 \\ 
&\qquad\qquad -9\cdot (-45)\cdot 255 +36\cdot 1\cdot 511
+ (-55)\cdot 9330 - 9\cdot 1\cdot 28501 + 1\cdot 611501 \\
&=\oldstylenums{63{,}173}\saf.
\end{align*}

We prove Theorem~\ref{thm:M0nbarGC} by studying the generating function
\begin{equation}\label{eq:genfun}
\hM:=1+z+\sum_{n\ge 3} [\ocaM_{0,n}]\frac{z^{n-1}}{(n-1)!}\saf.
\end{equation}

Recursive formulas for the Betti numbers and the Poincar\'e polynomial of
$\ocaM_{0,n}$ have been known for three decades, since S.~Keel's seminal 
work~\cite{MR1034665}. E.~Getzler (\cite{MR1363058}) and Y.~Manin (\cite{MR1363064}) 
obtained explicit functional and differential equations satisfied by the generating 
function~\eqref{eq:genfun}, and this information was interpreted in terms of the Grothendieck
class of~$\ocaM_{0,n}$ in~\cite{MR3701904}. We obtain the following explicit expressions
for this generating function.

\begin{theorem}\label{thm:main}
\begin{align*}
\hM&=
\sum_{\ell\ge 0} \frac{(\ell+1)^\ell}{(\ell+1)!} \left(
\left((1-\Lbb)(1+(z+1)\Lbb)\right)^{\frac {1+\ell}\Lbb-\ell}\,
\prod_{j=0}^{\ell-1} \left(1-\frac{j\Lbb}{\ell+1}\right)\right) \Lbb^\ell \\
&=\sum_{\ell\ge 0}\left(\sum_{k\ge 0}\frac{(\ell+1)^{\ell+k}}{(\ell+1)! k!}
(z-\Lbb-z\Lbb)^k \prod_{j=0}^{\ell+k-1} \left(1-\frac{j\Lbb}{\ell+1}\right)\right) \Lbb^\ell\saf.
\end{align*}
\end{theorem}

We prove Theorem~\ref{thm:M0nbarGC} as a direct consequence of the second expression, 
and the second expression follows from the first. We prove that the first expression equals~$\hM$
by showing that it is the solution of the aforementioned differential equation characterizing~$\hM$.

The first expression may be used to obtain alternative formulas for $[\ocaM_{0,n}]$
and for the Betti numbers.

\begin{corol}\label{cor:PP}
For all $n\ge 3$:
\begin{equation}\label{eq:Gc}
[\ocaM_{0,n}]=\sum_{\ell\ge 0} 
\frac{(1-\Lbb^2)^{\frac{\ell+1}\Lbb-\ell}}{(1+\Lbb)^{n-1}} \,
\left(\prod_{j=0}^{\ell+n-2} (\ell+1-j\Lbb)\right)\,
\frac{\Lbb^\ell}{(\ell+1)!}\saf.
\end{equation}
\end{corol}
As in the case of~\eqref{eq:M0nbarGC}, the r.h.s.~of~\eqref{eq:Gc} is, despite appearances, 
a polynomial in $\Lbb$ with positive integer coefficients and degree $n-3$. In particular,
for every $n$, only finitely many summands of~\eqref{eq:Gc} need be computed in order 
to determine $[\ocaM_{0,n}]$. For example, 
here are the series expansions of the first few summands of~\eqref{eq:Gc}, for $n=6$:
\begin{alignat*}{14}
& \ell=0: \quad && 1 && -16 &&\Lbb && +\frac{231}2 &&\Lbb^2 && -\frac{3109}6 &&\Lbb^3 
&& +\frac{40549}{24} &&\Lbb^4 && -\frac{265223}{60} &&\Lbb^5 && 
+\frac{7126141}{720} &&\Lbb^6+\cdots\\
& \ell=1: \quad &&  && \phantom{\,+\,}32&&\Lbb && -464&&\Lbb^2 && +3256 &&\Lbb^3 && 
-\frac{45326}3 &&\Lbb^4 && +\frac{158768}3 &&\Lbb^5 && -\frac{2288308}{15} &&\Lbb^6 + \cdots\\
& \ell=2: \quad && && && && \phantom{\,+\,}\frac{729}2 &&\Lbb^2 && -\frac{10935}2 &&\Lbb^3 && 
+\frac {163215}4 &&\Lbb^4 && -\frac{410265}2 &&\Lbb^5 && +\frac{12663747}{16} &&\Lbb^6 
+ \cdots \\
& \ell=3: \quad && && && && && && \phantom{\,+\,}\frac{8192}3 &&\Lbb^3 && 
-\frac{131072}3 &&\Lbb^4 && 
+\frac {1057792}3&&\Lbb^5 && -\frac{17410048}9 &&\Lbb^6 + \cdots \\
& \ell=4: \quad && && && && && && && && \phantom{\,+\,}\frac{390625}{24} &&\Lbb^4 && 
-\frac {3359375}{12}&&\Lbb^5 && +\frac{117453125}{48} &&\Lbb^6 + \cdots \\
& \ell=5: \quad && && && && && && && && && && 
\phantom{\,+\,}\frac {419904}5&&\Lbb^5 && -\frac{7768224}5 &&\Lbb^6 + \cdots \\
& \ell=6: \quad && && && && && && && && && && && 
&& \phantom{\,+\,}\frac{282475249}{720} &&\Lbb^6 + \cdots \\
\intertext{and their sum:}
& && 1 && +16&&\Lbb && +\quad 16&&\Lbb^2 && +\quad &&\Lbb^3 && +\quad 0&&\Lbb^4 && 
+\quad\,\, 0&&\Lbb^5 && +\qquad 0&&\Lbb^6+\cdots\saf.
\end{alignat*}
The infinite sum~\eqref{eq:Gc} converges to $[\ocaM_{0,6}]=1+16\Lbb+16\Lbb^2+\Lbb^3$
in the evident sense. Concerning Betti numbers, we have the following alternative
to Corollary~\ref{cor:bettisim}.

\begin{corol}\label{cor:betticom}
For $i>0$, let
\[
C_{nk i}=\frac{(-1)^i (2k i+ni+k +n-1)+k -i}{i(i+1)}
-\frac 1{i(k +1)^i}\sum_{j=0}^{k +n-2} j^i\saf.
\]
Then for all $\ell\ge 0$ and all $n\ge 3$ we have
\begin{equation}\label{eq:betti}
\rk H^{2\ell}(\ocaM_{0,n})
=\sum_{k=0}^\ell \frac{(k+1)^{k+n-1}}{(k+1)!}
\sum_{m=0}^{\ell-k} \frac 1{m!} 
\sum_{i_1+\cdots+i_m=\ell-k} C_{nki_1}\cdots C_{nki_m}
\saf.
\end{equation}
\end{corol}

By Faulhaber's formula, the numbers $C_{nki}$ may be written in terms of Bernoulli
numbers:
\begin{multline*}
C_{nk i}=\frac 1{i(i+1)}\bigg(
(-1)^i(2ki+ni+k+n-1)+k-i \\
-\frac 1{(k+1)^i}
\sum_{j=0}^i \binom{i+1}j B_j(k+n-1)^{i-j+1}
\bigg)\saf.
\end{multline*}
Thus, Corollary~\ref{cor:betticom} expresses the Betti numbers of $\ocaM_{0,n}$ as
certain combinations of Bernoulli numbers, just as Corollary~\ref{cor:bettisim} 
expresses the same in terms of Stirling numbers. For instance,
\begin{align*}
\rk H^6(\ocaM_{5,0}) &=\frac{271}2 -\frac{889}3 B_0 -\frac{277}3 B_1 - 8B_2 - \frac 43 B_3 \\
&\quad + \frac{1427}6 B_0^2 + \frac{590}3 B_0 B_1 + 16B_0 B_2+ 42 B_1^2+ 8B_1 B_2 \\
&\quad -\frac{256}3 B_0^3 - 128 B_0^2 B_1 - 64 B_0 B_1^2 - \frac{32}3 B_1^3 
\\
&=0
\intertext{while}
\rk H^6(\ocaM_{10,0}) &=\frac{756667}2 -\frac{5729797}{12} B_0 -86825 B_1 
- \frac{1271}2 B_2 - 3 B_3 \\
&\quad+ \frac{313439}2 B_0^2 + \frac{252355}4 B_0 B_1 + \frac{729}4 B_0 B_2
+ \frac{12719}2 B_1^2+ \frac{81}2 B_1 B_2 \\
&\quad-\frac{177147}{16} B_0^3 - \frac{59049}8 B_0^2 B_1 - \frac{6561}4 B_0 B_1^2 
- \frac{243}2 B_1^3 \\
&=\oldstylenums{63{,}173}\saf.
\end{align*}
For $n=1,2$, the right-hand side
of~\eqref{eq:betti} equals $1$ for $k=0$ and $0$ for $k>0$.
For every $n\ge 1$, \eqref{eq:betti} may be viewed as an infinite collection of 
identities involving Bernoulli numbers, equivalent to the corresponding identities
involving Stirling numbers arising from Corollary~\ref{cor:bettisim}. It appears to
be useful to have explicit expressions of both types. 

Keel's recurrence relation for the Betti numbers of $\ocaM_{0,n}$ can itself be viewed
as a sophisticated identity of Stirling numbers (by Corollary~\ref{cor:bettisim})
or Bernoulli numbers (by Corollary~\ref{cor:betticom}). Verifying such identities 
directly would provide a more transparent proof of these results, but the combinatorics 
needed for this verification seems substantially more involved than the somewhat indirect 
way we present in this paper.\smallskip

The proof of Theorem~\ref{thm:main} is presented in~\S\ref{sec:Mproof}. It relies crucially
on two combinatorial identities, both of which follow from `Lagrange inversion'. Proofs
of these identities are given in an appendix. Theorem~\ref{thm:M0nbarGC} and the
other corollaries stated above are proved in~\S\ref{sec:PPbettiproof}.

In \S\S\ref{sec:polp}--\ref{sec:Wm1} we study the finer structure of the formulas obtained
in the first part of the paper, building upon work carried out in~\cite{ACM}. This 
information is both interesting in itself and leads to further results on the Betti numbers
of $\ocaM_{0,n}$ and on the generating function $\hM$, see Theorem~\ref{thm:betti2}
and~\ref{thm:Mtree} below. As in~\cite{ACM}, 
we consider the generating function
\[
\alpha_k(z)
=\sum_{n\ge 3} \rk H^{2k}(\ocaM_{0,n})\frac{z^{n-1}}{(n-1)!} 
\]
for the coefficients of $\Lbb^k$ in $[\ocaM_{0,n}]$, i.e., the individual Betti numbers of
$\ocaM_{0,n}$. While Corollary~\ref{cor:bettisim} yields an explicit expression for
the coefficients of this generating function, there is an interesting structure associated
with $\alpha_k(z)$ that is not immediately accessible from such an expression.
Specifically, by~\cite[Theorem~4.1]{ACM} 
there exist polynomials $p_m^{(k)}(z)\in \Qbb[z]$, $0\le m\le k$, such that
\begin{equation}\label{eq:alphak}
\alpha_k(z) = e^z\sum_{m=0}^k (-1)^m p_m^{(k)}(z)\, e^{(k-m) z}
\end{equation}
for all $k\ge 0$.
 It is proved in~\cite{ACM} that $p_m^{(k)}(z)$ has degree $2m$
and positive leading coefficient, and that
\begin{equation}\label{eq:pk0}
p^{(k)}_0=\frac{(k+1)^k}{(k+1)!}\saf,
\end{equation}
and this is used to establish an asymptotic form of log-concavity of $[\ocaM_{0,n}]$. 
Several polynomials $p^{(k)}_m(z)$ are computed explicitly in~\cite{ACM}; for instance,
\[
p^{(1)}_1(z)=1+z+\frac{z^2}2\saf,
\]
so that
\[
\alpha_1(z)=e^z\left(p^{(1)}_0e^z-p^{(1)}_1(z)\right)=e^z\left(e^z-1-z-\frac{z^2}2\right)
=\frac{z^3}{3!}+5\frac{z^4}{4!}+16\frac{z^5}{5!}+42\frac{z^6}{6!}+\cdots
\]
is the generating function for $\rk H^2(\ocaM_{0,n})$. On the basis of extensive computations,
in~\cite{ACM} we proposed the following.

\begin{conj}\label{conj:plogc}
For all $k\ge 1$, the polynomials $p_m^{(k)}$, $m=1,\dots,k$, have positive coefficients
and are log-concave with no internal zeros. All but $p^{(1)}_1$, $p^{(3)}_3$, 
$p^{(5)}_5$ are ultra-log-concave.
\end{conj}

In this paper we prove the first statement in this conjecture and provide substantial
numerical evidence for the second part. For this purpose, we assemble the 
polynomials~$p_m^{(k)}(z)$ in the generating function
\[
P(z,t,u):=
\sum_{m\ge 0}\sum_{\ell\ge 0} p^{(m+\ell)}_m(z) t^\ell u^m\saf.
\]
By~\eqref{eq:alphak}, $\hM(z,\Lbb)$ is (up to an exponential factor) a specialization of
$P(z,t,u)$. 
This fact and Theorem~\ref{thm:main} may be used to obtain an explicit expression for
$P(z,t,u)$: we prove the following statement in~\S\ref{sec:polp}.

\begin{theorem}\label{thm:main2}
\[
P(z,t,u) =\sum_{\ell\ge 0} \left(
\frac{(\ell+1)^\ell}{(\ell+1)!} e^{-(\ell+1)z}\,
\left((1+u)(1-u(z+1))\right)^{-\frac{1+\ell(u+1)}u}\,
\prod_{j=0}^{\ell-1} \left(1+\frac {ju}{\ell+1}\right)\right) t^\ell\saf.
\]
\end{theorem}

In~\S\ref{sec:pkm} we use Theorem~\ref{thm:main2} to obtain more information on the 
polynomials $p^{(k)}_m(z)$. Denote by~$c^{(k)}_{mj}$ the coefficients of $p^{(k)}_m(z)$, 
so that $p^{(k)}_m(z)=\sum_{j=0}^{2m} c^{(k)}_{mj} z^j$.

\begin{theorem}\label{thm:pcG}
For all $m\ge 0$, $0\le j\le 2m$, 
there exist polynomials $\Gamma_{mj}(\ell)\in \Qbb[\ell]$ of degree $2m-j$ such that
\begin{equation}\label{eq:cfromGamma}
c^{(k)}_{mj} =\frac{(k-m+1)^{k-2m+j}}{(k-m+1)!} \cdot \Gamma_{mj}(k-m)
\end{equation}
for all $k\ge m$.
Further $\Gamma_{mj}(\ell)>0$ for all $m\ge 0$, $0\le j\le 2m$, $\ell\ge 0$.
\end{theorem}

The polynomial $\Gamma_{mj}(\ell)$ are effectively computable. For example, 
\[
\Gamma_{4,0}= \frac{27}{128}\ell^8 + \frac{57}{32}\ell^7 + \frac{4295}{576}\ell^6 + 
\frac{1341}{80}\ell^5 + \frac{28867}{1152}\ell^4 + \frac{2143}{96}\ell^3 
+ \frac{3619}{288}\ell^2 + \frac{119}{30}\ell + \frac{13}{24}\saf. 
\]
The first several hundred such polynomials have positive coefficients and are in fact 
log concave. However, $\Gamma_{20,0}$ is not log concave, and the coefficient of
$\ell^2$ in the degree-$42$ polynomial $\Gamma_{21,0}$, namely, 
$-\frac{97330536888617758406393}{2248001455555215360000}$, is {\em negative,\/}
a good reminder of how delicate these notions are and a cautionary tale about making
premature conjectures.

Nevertheless, as stated in Theorem~\ref{thm:pcG}, we can prove that $\Gamma_{mj}(\ell)$ 
is positive for all $m$, $0\le j\le 2m$, $\ell\ge 0$, and this has the following immediate
consequence.

\begin{corol}
The polynomials $p^{(k)}_m(z)$ have positive coefficients.
\end{corol}

This proves part of Conjecture~\ref{conj:plogc}. Further, we obtain substantial evidence
for the rest of the conjecture, dealing with log-concavity of the polynomials~$p^{(k)}_m(z)$.
Specifically, we reduce the proof of ultra-log-concavity of $p^{(k)}_m(z)$ for a fixed $m$
and all $k\ge m$ to a finite computation involving the polynomials $\Gamma_{mj}$. 
A few hours of computing time verified the conjecture for $m=1,\dots, 100$.

As a byproduct of these considerations, we obtain an alternative expression for the Betti 
numbers of $\ocaM_{0,n}$, in terms of the polynomials $\Gamma_{mj}(\ell)$.

\begin{theorem}\label{thm:betti2}
For $n\ge 3$ and $0\le \ell\le n-3$:
\[
\rk H^{2\ell}(\ocaM_{0,n}) =
\sum_{k+m=\ell} 
(-1)^m \frac{(k+1)^{n-2+k-m}}{k!}
\sum_{j=0}^{2m}
(n-1)\cdots (n-j)\,
\Gamma_{mj}(k)
\saf.
\]
\end{theorem}

This formula generalizes directly the well-known formula for the second Betti number,
\[
\rk H^2(\ocaM_{0,n}) = \frac 12\cdot 2^n - \frac{n^2 - n + 2}2\saf,
\]
cf.~\cite[p.~550]{MR1034665}. For instance,
\begin{multline*}
\rk H^6(\ocaM_{0,n}) =\frac 23\cdot 4^n - \frac {(n+4)(n+3)}{12}\cdot 3^n 
+ \frac{3n^4+14n^3+57n^2+118n+96}{192}\cdot 2^n \\
- \frac{n^6-7 n^5+35 n^4 - 77 n^3 + 120 n^2 - 72 n + 32}{48}\saf,
\end{multline*}
giving
\[
\rk H^6(\ocaM_{0,10}) =\frac 23\cdot 4^{10}-\frac {91}6\cdot 3^{10}+\frac{531}2\cdot 2^{10} 
-\frac{73039}2=\oldstylenums{63{,}173}\saf.
\]

Studying the function $P(z,t,u)$ also reveals an intriguing connection with the {\em Lambert 
W-function,\/} which we explore in~\S\ref{sec:Wm1}.
Recall that the Lambert W-function $W(t)$ is characterized by the identity
$W (t)e^{W (t)} = t$; the reader is addressed to~\cite{MR1414285} for a detailed treatment 
of this function. We consider in particular the `tree function' $T(t)=-W(-t)$, where 
$W$ is the principal branch of the Lambert W-function. (This function owes its name to the fact 
that $T(t)=\sum_{n\ge 1}\frac{T_n t^n}{n!}$, where $T_n=n^{n-1}$ is the number of rooted 
trees on $n$ labelled vertices.)
For a fixed $m$, consider the generating function
\begin{equation}\label{eq:gfPm}
P_m(z,t)=\sum_{\ell\ge 0} p^{(m+\ell)}_m(z) t^\ell\saf,
\end{equation}
that is, the coefficient of $u^m$ in $P$.
This is a polynomial in $\Qbb[[t]][z]$ of degree $2m$. Explicit computations show that
\begin{align*}
P_0 &=e^T \\
P_1 &=\frac {e^T}{(1-T)}\left(\frac 12z^2+(1+T)z+\frac 12 (2+T^2)\right)
\end{align*}
and
\begin{multline*}
P_2=\frac {e^T}{(1-T)^3}\left(
\frac 18 z^4+\frac{5+2T-T^2}6 z^3 
+\frac{8+4T+T^2-2T^3}4 z^2 \right.\\\left.
+\frac{4+2T+T^2-T^4}2 z
+\frac{12+24T-12T^2+8T^3-T^4-4T^5}{24}
\right)
\end{multline*}
where $T=T(t)$ is the tree function.
(In fact, the first expression is a restatement of~\eqref{eq:pk0}.)
For all $m\ge 0$, the series $e^{-T} P_m(z,t)$ is a polynomial in $z$ with coefficients
in $\Qbb[[t]]$. We prove that for all $m\ge 0$ these coefficients can be expressed as 
{\em rational functions\/} in the tree function $T$, as in the examples shown above. 
More precisely, 

\begin{prop}\label{prop:treeprop}
Let $T=T(t)=-W(-t)$ be the tree function.
For $m>0$, there exist polynomials $F_m(z,\tau)\in \Qbb[z,\tau]$, of degree $2m$ in $z$
and $< 3m$ in $\tau$, such that
\begin{equation}\label{eq:Pmtree}
P_m(z,t)=e^T \frac{F_m(z,T)}{(1-T)^{2m-1}}\saf.
\end{equation}
\end{prop}

Setting $F_0(z,T)=\frac 1{1-T}$ extends the validity of~\eqref{eq:Pmtree} to the case $m=0$.

Since for $m>0$ the polynomials $F_m(z,\tau)$ have degree $<3m$ in $\tau$, they are
characterized by the congruence
\[
F_m(z,\tau)\equiv \sum_{\ell=0}^{3m-1} p_m^{(m+\ell)}(z) (1-\tau)^{2m-1} e^{-(\ell+1)\tau} 
\tau^\ell \mod \tau^{3m}\saf.
\]

Proposition~\ref{prop:treeprop} implies the following alternative expression for $\hM$.

\begin{theorem}\label{thm:Mtree}
Let $T=T(e^z \Lbb)$, where $T$ is the tree function. Then with $F_m(z,\tau)$ as above
we have
\begin{equation}\label{eq:Mtree}
\hM=\frac T\Lbb\, \sum_{m\ge 0} \frac{(-1)^m F_m(z,T) }{(1-T)^{2m-1}}\,\Lbb^m \saf,
\end{equation}
\end{theorem}

While the identities stated in Theorem~\ref{thm:main} are more explicit expressions
for $\hM$, they show no direct trace of the relation with the principal branch of the 
Lambert W-function displayed in Theorem~\ref{thm:Mtree}. The proof of this result 
relies on another combinatorial identity involving Stirling numbers,
whose proof is also given in the appendix.\smallskip

It is essentially evident that~\eqref{eq:M0nbarGC} may be expressed in terms of a 
product of matrices. We formalize this remark in~\S\ref{sec:Stirma}, by showing that
the class $[\ocaM_{0,n}]$ may be recovered as a generalized trace of a matrix 
obtained in a very simple fashion from the standard matrices defined by Stirling 
numbers of first and second kind (Theorem~\ref{thm:Stma}). This is simply a restatement
of~Theorem~\ref{thm:M0nbarGC}, but it may help to relate the results of this paper with
the extensive literature on Stirling numbers.

There is another connection between $\hM_{0,n}$ and the Lambert W-function: the
generating function for the Euler characteristic of $\hM_{0,n}$ may be expressed in
terms of the other real branch of the Lambert W-function, denoted $W_{-1}$ 
in~\cite{MR1414285}. The various expressions obtained in this paper for the Betti 
numbers imply explicit formulas for the individual Euler characteristics $\chi(\hM_{0,n})$.
In~\S\ref{sec:OtEcoM} we highlight one compelling appearance of $\chi(\hM_{0,n})$ 
as the leading coefficient of a polynomial determined by a sum of products of
Stirling numbers, see Proposition~\ref{prop:Eucha}.\smallskip

We end this introduction by pointing out that another expression for~$[\ocaM_{0,n}]$
may be obtained as a consequence of the result of Ezra Getzler 
mentioned earlier, which we reproduce as follows. For consistency with the notation
used above, we interpret Getzler's formula with $\Lbb=t^2$.

\begin{theorem}[Getzler \cite{MR1363058}]\label{thm:Getz}
Let
\[
g(x,\Lbb) =x - \sum_{n=2}^\infty \frac{x^n}{n!} \sum_{i=0}^{n-2}(-1)^i \Lbb^{(n-i-2)} 
\dim H_i(\caM_{0,n+1}) 
=x-\frac{(1+x)^\Lbb - (1+\Lbb x)}{\Lbb(\Lbb-1)}\saf.
\]
Then
\begin{align*}
f(x,\Lbb):=x+\sum_{n=2}^\infty \frac{x^n}{n!} \sum_{i=0}^{n-2} \Lbb^i \dim H_{2i}(\ocaM_{0,n+1})
\end{align*}
is the inverse of $g$, in the sense that $f(g(x,\Lbb),\Lbb)=x$.
\end{theorem}

Getzler's generating function~$f$ differs from our $\hM$ by the absence of the constant
term. Theorem~\ref{thm:Getz} may be seen to be equivalent to the functional equation
for the Poincar\'e polynomial mentioned earlier, also proved (with different notation) in~\cite{MR1363064}. 

We can apply Lagrange inversion (see e.g., the first formula in \cite[p.~146, (18)]{MR231725}) 
to the identity $f(g(x,\Lbb),\Lbb)=x$. Specifically, assume that
\[
f(y)=\sum_{n\ge 0} a_n\frac{y^n}{n!}
\]
is a power series such that
\[
\alpha(x)=f(\beta(x))=\sum_{n\ge 0} a_n \frac{\beta(x)^n}{n!}
\]
for constants $a_n$ and functions $\alpha(x)$, $\beta(x)$ such that~$\beta(0)=0$. 
Then Lagrange inversion states that
\[
a_n=\frac{d^{n-1}}{dx^{n-1}}(\alpha'(x) \varphi(x)^n)|_{x=0}\saf,
\]
where $\varphi(x)=\frac x{\beta(x)}$. Applying this formula with $a_n=[\ocaM_{0,n+1}]$, 
$\alpha(x)=x$, $\beta(x)=g(x,\Lbb)$ gives 
\[
[\ocaM_{0,n+1}]=(n-1)!\cdot \text{coefficient of $x^{n-1}$ in the expansion of }
\left(\frac x{g(x,\Lbb)}\right)^n
\]
with $g(x,\Lbb)$ as in Theorem~\ref{thm:Getz}. Therefore,
\[
[\ocaM_{0,n}] = (n-2)!\cdot \text{coefficient of $x^{n-2}$ in the expansion of }
\left(\frac {\Lbb(\Lbb-1)x}{1+\Lbb^2x- (1+x)^\Lbb}\right)^{n-1}
\saf.
\]
This also easily implies that
\[
\hchi(\ocaM_{0,n})=(n-2)!\cdot \text{coefficient of $x^{n-2}$ in the expansion of }
\left(\frac 1{2-\frac {1+x}x \log(1+x)}\right)^{n-1}\saf.
\]
We do not know if alternative proofs of the results in this paper may be obtained 
from these expressions.\smallskip

%%%

{\em Acknowledgments.} 
P.A.~was supported in part by the Simons Foundation, collaboration 
grant~\#625561, and thanks Caltech for the hospitality. M.M.~was
supported by NSF grant DMS-2104330. E.N.~was supported
by a Summer Undergraduate Research Fellowship at Caltech
funded in part by a generous gift from Marcella Bonsall. 

%%%

\section{Proof of Theorem~\ref{thm:main}}\label{sec:Mproof}

Throughout the paper, all functions are implicitly considered as formal power series, used as
generating functions for the coefficients in their expansion. 

A recursion determining $[\ocaM_{0,n}]$ or, equivalently, the Poincar\'e polynomial of 
$\ocaM_{0,n}$, is well-known: see~\cite{MR1034665}, \cite{MR1363064}, \cite{MR3701904}.
It is equivalent to a differential equation for the generating function
\[
\hM(z,\Lbb):=1+z+\sum_{n\ge 3} [\ocaM_{0,n}] \frac{z^{n-1}}{(n-1)!}\saf,
\]
namely,
\begin{equation}\label{eq:Mdifeq}
\frac{\partial \hM}{\partial z}=\frac{\hM}{1+\Lbb(1+z)-\Lbb \hM}\saf,
\end{equation}
subject to the initial condition $\hM|_{z=0}=1$. (This is a restatement of~\cite[(0.8)]{MR1363064}.)

Denote by $\oM$ the first expression stated in Theorem~\ref{thm:main}:
\begin{equation}\label{eq:Mgfonly}
\oM:=\sum_{\ell\ge 0} \frac{(\ell+1)^\ell}{(\ell+1)!} 
\bigg((1-\Lbb)(1+(z+1)\Lbb)\bigg)^{\frac {\ell+1}\Lbb -\ell}\,
\prod_{j=0}^{\ell-1} \left(1-\frac{j\Lbb}{\ell+1}\right) \Lbb^\ell\saf,
\end{equation}
where we view $\Lbb$ as an indeterminate.
We need to verify that $\hM=\oM$. To prove this result, it suffices
to verify that $\oM$ satisfies~\eqref{eq:Mdifeq}, which we rewrite as
\begin{equation}\label{eq:deoM}
(1+\Lbb(1+z))\frac{\partial\oM}{\partial z}-\oM=\Lbb\, \frac{\partial \oM}{\partial z}\, \oM\saf,
\end{equation}
and the initial condition~$\oM|_{z=0}=1$.

Both of these claims will be reduced to combinatorial identities. 
For the first, it is convenient to separate the part of~$\oM$ depending on $z$ from 
the rest: write
\[
\oM=\sum_{k=0}^\infty F(k) (1+(z+1)\Lbb)^{\frac{1+k(1-\Lbb)}\Lbb}
\]
with
\[
F(k):=\frac{(k+1)^k}{(k+1)!} \cdot
\prod_{j=0}^{k-1} \left(1-\frac{j\Lbb}{k+1}\right) (1-\Lbb)^{\frac{1+k(1-\Lbb)}\Lbb}
\Lbb^k\saf.
\]
We have
\[
\frac{\partial\oM}{\partial z}
=\sum_{k=0}^\infty F(k) (1+(z+1)\Lbb)^{\frac{1+k(1-\Lbb)}\Lbb}
\frac{1+k(1-\Lbb)}{(1+(z+1)\Lbb)}\saf,
\]
implying
\[
(1+(z+1)\Lbb)\frac{\partial\oM}{\partial z}-\oM
=\sum_{k=1}^\infty F(k) k (1-\Lbb) (1+(z+1)\Lbb)^{\frac{1+k(1-\Lbb)}\Lbb}\saf.
\]
In order to verify~\eqref{eq:deoM}, we need to verify that this expression equals
\begin{align*}
\Lbb \, \frac{d\oM}{dz}\, \oM &=
\Lbb\sum_{\ell,m}  
F(\ell) (1+(z+1)\Lbb)^{\frac{1+\ell(1-\Lbb)}\Lbb} 
\frac{1+\ell(1-\Lbb)}{(1+(z+1)\Lbb)}
F(m) (1+(z+1)\Lbb)^{\frac{1+m(1-\Lbb)}\Lbb} \\
&=\Lbb\sum_{\ell,m}  
F(\ell) F(m) (1+\ell(1-\Lbb))(1+(z+1)\Lbb)^{\frac{2+(\ell+m)(1-\Lbb)}\Lbb-1} \\
&= \Lbb\sum_{k\ge 1} \sum_{\ell+m=k-1}  
F(\ell) F(m) (1+\ell(1-\Lbb))(1+(z+1)\Lbb)^{\frac{1+k(1-\Lbb)}\Lbb} \saf.
\end{align*}
Therefore, In order to prove that $\oM$ satisfies~\eqref{eq:deoM},
it suffices to prove that for all~$k\ge 1$
\[
\Lbb \sum_{\ell+m=k-1} F(\ell) F(m)\cdot (1+\ell(1-\Lbb)) = F(k)\cdot k(1-\Lbb) \saf.
\]
To simplify this further, let
\[
E(\ell):=\frac{(\ell+1)^\ell}{(\ell+1)!} \cdot \prod_{j=0}^{\ell-1} \left(1-\frac{j\Lbb}{\ell+1}\right)\saf,
\]
so that
\[
F(\ell)=E(\ell)\cdot  (1-\Lbb)^{\frac{1+\ell(1-\Lbb)}\Lbb} \Lbb^\ell\saf.
\]
The sought-for identity is
\begin{multline*}
\Lbb \sum_{\ell+m=k-1} E(\ell)\cdot  (1-\Lbb)^{\frac{1+\ell(1-\Lbb)}\Lbb} \Lbb^\ell
E(m)\cdot  (1-\Lbb)^{\frac{1+m(1-\Lbb)}\Lbb} \Lbb^m\cdot (1+\ell(1-\Lbb))  \\
= E(k)\cdot  (1-\Lbb)^{\frac{1+k(1-\Lbb)}\Lbb} \Lbb^k \cdot k(1-\Lbb) \saf.
\end{multline*}
Clearing the common factor proves the following.

\begin{claim}\label{claim:fire}
Let
\[
E(\ell):=\frac{(\ell+1)^\ell}{(\ell+1)!} \cdot \prod_{j=0}^{\ell-1} \left(1-\frac{j\Lbb}{\ell+1}\right)\saf.
\]
Then in order to prove that $\oM$ satisfies~\eqref{eq:deoM}, it suffices to prove that for all
$k\ge 1$,
\begin{equation}\label{eq:mainid}
\sum_{\ell+m=k-1} E(\ell) E(m)\cdot (1+\ell(1-\Lbb)) = k\cdot E(k) \saf.
\end{equation}
\end{claim}

For every $k$, this is an identity of polynomials in $\Lbb$. For example, for $k=3$ 
it states that
\[
\left(\frac 32-\frac\Lbb 2\right)(3-2\Lbb)
+(2-\Lbb)
+\left(\frac 32-\frac \Lbb 2\right)
=8\left(1-\frac \Lbb 4\right)\left(1 - \frac\Lbb 2\right)\saf.
\]
As such, \eqref{eq:mainid} is equivalent to the identity obtained by performing an 
invertible change of variables. Setting $\Lbb=-\frac 1{w-1}$, 
\begin{align*}
E(\ell) &= \frac 1{(\ell+1)!} \prod_{j=0}^{\ell-1} (\ell+1-j\Lbb)
= \frac 1{(\ell+1)!} \prod_{j=0}^{\ell-1} \frac{(w-1)(\ell+1)+j}{w-1} \\
&= \frac {\prod_{j=0}^{\ell-1} ((w-1)(\ell+1)+j)}{(\ell+1)!(w-1)^\ell}  \saf,
\end{align*}
and trivial manipulations show that~\eqref{eq:mainid} is then equivalent to the
identity
\begin{multline*}
\sum_{\ell+m=k-1}
\frac {\prod_{j=1}^{\ell} ((w-1)(\ell+1)+j)}{\ell!}\cdot
\frac {\prod_{j=0}^{m-1} ((w-1)(m+1)+j)}{m!} \cdot \frac 1{m+1}\\
=\frac {\prod_{j=1}^{k-1} ((w-1)(k+1)+j)}{(k-1)!}
\end{multline*}
of polynomials in $\Qbb[w]$. In order to verify this identity, it suffices to verify that it
holds when $w$ is evaluated at infinitely many integers. We can then restate
Claim~\ref{claim:fire} as follows. 

\begin{claim}
In order to verify that $\oM$ satisfies~\eqref{eq:deoM}, it suffices to prove the following
binomial identity
\begin{equation}\label{eq:binid1}
\sum_{\ell+m=k-1}
\binom{w(\ell+1)-1}\ell \cdot
\binom{w(m+1)-2}m 
\cdot \frac 1{m+1}
=\binom{w(k+1)-2}{k-1}
\end{equation}
for all positive integers $k$ and $w$.
\end{claim}

This identity indeed does hold. It may be obtained as a specialization of more general, 
known, identities; see e.g., \cite[p.~169]{MR231725} or \cite{MR75170}, where such
identities are identified as generalizations of `Vandermonde's convolution'. 
We include a proof of~\eqref{eq:binid1} in the appendix, Lemma~\ref{lem:binid1},
and this concludes the verification that $\oM$ satisfies the differential 
equation~\eqref{eq:deoM}.

Next, we need to verify that $\oM$ satisfies the same initial condition as $\hM$, i.e., 
$\oM|_{z=0}=1$, that is,
\begin{equation}\label{eq:inico}
\sum_{\ell\ge 0} \frac{(\ell+1)^\ell}{(\ell+1)!} 
(1-\Lbb^2)^{\frac {1+\ell}\Lbb-\ell}\,
\prod_{j=0}^{\ell-1} \left(1-\frac{j\Lbb}{\ell+1}\right) \Lbb^\ell = 1\saf.
\end{equation}
This statement is surprisingly nontrivial. 

\begin{remark}
Identity~\eqref{eq:inico} may be viewed as an infinite collection of identities involving 
Bernoulli numbers. Explicitly, for $k\ge 0$ and $i\ge 1$ let
\[
A_{k i}:=\begin{cases}
-\dfrac{2(k+1)}{i+1} -\dfrac 1{i(i+1)(k+1)^i}\sum_{j=0}^i \binom{i+1}j k^{i-j+1} B_j 
&\quad\text{$i$ odd $>0$}\\
\hfill\dfrac{2k}i\hfill -\dfrac 1{i(i+1)(k+1)^i}\sum_{j=0}^i \binom{i+1}j k^{i-j+1} B_j 
&\quad\text{$i$ even $>0$}\saf.
\end{cases}
\]
Identity~\eqref{eq:inico} is then equivalent to the assertion that
for all positive integers $\ell$
\[
\sum_{k=0}^\ell
\frac{(k+1)^k}{(k+1)!} \sum_{m\ge 0} \frac 1{m!} 
 \sum_{i_1+\cdots+i_m=\ell-k} A_{k i_1}\cdots A_{k i_n} = 0
\]
where all indices $i_j$ in the summation are positive integers. (This computation will
be carried out in a more general setting in the proof of Corollary~\ref{cor:betticom}.)
For increasing values of~$\ell$, this identity states that
\begin{align*}
B_1 &= -\frac 12 B_0\\
B_2 &= 4 - \frac{13}3 B_0- B_1 + \frac 14 (B_0+2B_1)^2 \\
B_3 &= 12 - \frac{259}{12} B_0- 15B_1 + \frac{1}2 B_2
+ \frac{27}4 B_0^2+ \frac{45}4 B_0 B_1\\
&\qquad\qquad\qquad
+ \frac{3}4 B_0 B_2+ \frac{7}2 B_1^2+ \frac{3}2 B_1 B_2-\frac 1{16} (B_0+2B_1)^3
\end{align*}
etc. Our verification of~\eqref{eq:inico} proves all these identities simultaneously.
\qede\end{remark}

For increasing values of $\ell$, the summands in the left-hand side~of~\eqref{eq:inico} 
expand to
\begin{alignat*}{12}
& \ell=0: \quad && (1-\Lbb^2)^{\frac 1\Lbb} && = && 
1 && -\, &&\Lbb \,+\, && \frac 12 && \Lbb^2-\, && \frac 23 && \Lbb^3 && +\frac{13}{24} && \Lbb^4
+\cdots\\
& \ell=1: \quad && (1-\Lbb^2)^{\frac 2\Lbb-1}\Lbb && = && 
&&  && \Lbb \,-\, && 2 && \Lbb^2+\, && 3 && \Lbb^3 && -\frac {13}3 && \Lbb^4 + \cdots\\
& \ell=2: \quad && (1-\Lbb^2)^{\frac 3\Lbb-2}\frac{3-\Lbb}2 \Lbb^2 && = && 
&& && && \frac 32 && \Lbb^2 -\, && 5 && \Lbb^3 && +\frac{45}4 && \Lbb^4 + \cdots \\
& \ell=3: \quad && (1-\Lbb^2)^{\frac 4\Lbb-3}\frac{4-\Lbb}3\frac{4-2\Lbb}2 \Lbb^3 && = && 
&& && && && && \frac 83 && \Lbb^3 && -\frac{38}3 && \Lbb^4 + \cdots 
\end{alignat*}
The sum of these terms is a power series, and the task is to verify that this series is
the constant $1$. For this, it suffices to prove that the series has a limit of~$1$ for 
$\Lbb=\frac 1m$ for all integers~$m>1$. Setting $\Lbb=\frac 1m$ with $m>1$ an integer, 
the left-hand side of~\eqref{eq:inico} may be written
\begin{multline*}
\sum_{\ell\ge 0} \frac{(\ell+1)^\ell}{(\ell+1)!} 
\left(1-\frac 1{m^2}\right)^{(\ell+1)m-\ell}\cdot
\prod_{j=0}^{\ell-1} \left(1-\frac{j}{m(\ell+1)}\right) \frac 1{m^\ell} \\
= \left(1-\frac 1{m^2}\right)^m
\sum_{\ell\ge 0} 
\left(\frac {(m^2-1)^{m-1}}{m^{2m}}\right)^\ell 
\binom{m(\ell+1)}\ell \frac 1{\ell+1}
\end{multline*}
after elementary manipulations.
Therefore, the following assertion holds.

\begin{claim}
In order to verify~\eqref{eq:inico}, it suffices to verify that the identity
\begin{equation}\label{eq:binid2}
\sum_{\ell\ge 0} 
\frac 1{\ell+1}\binom{m(\ell+1)}\ell 
\left(\frac {(m^2-1)^{m-1}}{m^{2m}}\right)^\ell 
 = \left(1-\frac 1{m^2}\right)^{-m}
\end{equation}
holds for all integers $m>1$.
\end{claim}

Identity~\eqref{eq:binid2} follows from Lemma~\ref{lem:binid2}, which states
\[
x^\alpha = \sum_{\ell\ge 0} \frac \alpha {\alpha+\ell\beta}\binom{\alpha+ \ell\beta}\ell y^\ell
\]
for all positive integers $\alpha$, $\beta$, with $y=(x-1)x^{-\beta}$.
Indeed, setting $\alpha=\beta=m$, this identity gives
\[
x^m = \sum_{\ell\ge 0} \frac 1 {\ell+1}\binom{m(\ell+1)}\ell y^\ell\saf,
\]
and specializing to $x=\left(1-\frac 1{m^2}\right)^{-1}$ yields~\eqref{eq:binid2} as needed.
This concludes the proof of~\eqref{eq:inico}.\smallskip

Summarizing, we have proved that $\oM$ and $\hM$ satisfy the same differential equation
and initial conditions. Therefore $\hM=\oM$, confirming the first assertion in
Theorem~\ref{thm:main}.
\qed\smallskip

In order to complete the proof of Theorem~\ref{thm:main}, we need to verify that $\oM$
also equals the second expression in that statement:
\begin{equation}\label{eq:mainsim}
\oM=\sum_{\ell\ge 0}\sum_{k\ge 0}\frac{(\ell+1)^{\ell+k}}{(\ell+1)! k!}
(z-\Lbb-z\Lbb)^k \prod_{j=0}^{\ell+k-1} \left(1-\frac{j\Lbb}{\ell+1}\right) \Lbb^\ell\saf.
\end{equation}
For this, use the expansion
\[
(1+x)^\alpha = 1 + \alpha x + \frac{\alpha(\alpha-1)}2 x^2 
+ \frac{\alpha(\alpha-1)(\alpha-2)}{3!} x^3 + \cdots 
\]
to obtain
\begin{align*}
((1-\Lbb)(1+(z+1)\Lbb))^{\frac{\ell+1}\Lbb-\ell} 
&= (1+\Lbb(z-\Lbb-\Lbb z))^{\frac{\ell+1}\Lbb-\ell} \\
&=\sum_{k\ge 0} \frac 1{k!} (z-\Lbb-\Lbb z)^k 
\prod_{j=0}^{k-1}\left((1+\ell)-(\ell+j)\Lbb\right) \\
&=\sum_{k\ge 0} \frac {(\ell+1)^k}{k!} (z-\Lbb-\Lbb z)^k 
\prod_{j=0}^{k-1}\left(1-\frac{(\ell+j)\Lbb}{\ell+1}\right) \saf.
\end{align*}
Making use of this expression in~\eqref{eq:Mgfonly} yields
\[
\oM=
\sum_{\ell\ge 0}\sum_{k\ge 0}\frac{(\ell+1)^{\ell+k}}{(\ell+1)!k!}(z(1-\Lbb)-\Lbb)^k
\prod_{j=0}^{k-1}\left(1-\frac{(\ell+j)\Lbb}{\ell+1}\right)
\prod_{j=0}^{\ell-1}\left(1-\frac{j\Lbb}{\ell+1}\right)\Lbb^\ell
\]
and~\eqref{eq:mainsim} follows.
\qed

\begin{remark}
An alternative proof of Theorem~\ref{thm:main} can be obtained by using a different 
characterization of the generating function~$\hM$. Manin (\cite[Theorem~0.3.1, 
(0.7)]{MR1363064}) and Getzler (\cite[Theorem~5.9]{MR1363058}) prove that 
$\hM$ is the only solution of the functional equation
\begin{equation}\label{eq:funeqM}
\hM^\Lbb=\Lbb^2\hM + (1-\Lbb)(1+(z+1)\Lbb)\saf.
\end{equation}
It can be verified that the function~$\oM$ satisfies this equation, and this implies
$\oM=\hM$. As in the proof given above, the argument ultimately relies on
Lemma~\ref{lem:binid2}. 
\qede\end{remark}

%%%

\section{Grothendieck class and Betti numbers, and proof of 
Theorem~\ref{thm:M0nbarGC}}\label{sec:PPbettiproof}

Corollaries~\ref{cor:PP} and~\ref{cor:betticom} follow directly from the first expression 
shown in Theorem~\ref{thm:main}.

\begin{proof}[Proof of Corollary~\ref{cor:PP}]
According to Theorem~\ref{thm:main},
\[
1+z+\sum_{n\ge 3} [\ocaM_{0,n}] \frac{z^{n-1}}{(n-1)!} 
=\sum_{\ell\ge 0} \frac{(\ell+1)^\ell}{(\ell+1)!} 
\left((1-\Lbb)(1+(z+1)\Lbb)\right)^{\frac {1+\ell(1-\Lbb)}{\Lbb}}
\prod_{j=0}^{\ell-1} \left(1-\frac{j\Lbb}{\ell+1}\right) \Lbb^\ell\saf.
\]
Thus, the class $[\ocaM_{0,n}]$ may be obtained by setting $z=0$ in the $(n-1)$-st 
derivative of this expression with respect to $z$. A simple induction proves that
\[
\frac{d^{n-1}}{dz^{n-1}} \left((1+(z+1)\Lbb)^{\frac {1+\ell}\Lbb - \ell}\right)
=(1+(z+1)\Lbb)^{\frac {1+\ell}\Lbb - \ell-n+1}\prod_{j=0}^{n-2} (1+\ell-(\ell+j)\Lbb)\saf,
\]
and hence
\[
\left.\frac{d^{n-1}}{dz^{n-1}} \left((1+(z+1)\Lbb)^{\frac {1+\ell}\Lbb - \ell}\right)\right|_{z=0}
=(1+\Lbb)^{\frac {1+\ell}\Lbb - \ell-n+1}(\ell+1)^{n-1}\prod_{j=0}^{n-2} 
\left(1-\frac{(\ell+j)\Lbb}{\ell+1}\right)\saf.
\]
It follows that $[\ocaM_{0,n}]$ equals
\begin{equation}\label{eq:M0ninp}
\sum_{\ell\ge 0} \frac{(\ell+1)^{\ell+n-1}}{(\ell+1)!}
(1-\Lbb)^{\frac{1+\ell}\Lbb-\ell} 
(1+\Lbb)^{\frac {1+\ell}\Lbb - \ell-n+1}
\prod_{j=0}^{\ell+n-2} \left(1-\frac{j\Lbb}{\ell+1}\right)
\Lbb^\ell
\end{equation}
and the expression stated in Corollary~\ref{cor:PP} follows.
\end{proof}

\begin{proof}[Proof of Corollary~\ref{cor:betticom}]
The rank of $H^{2\ell}(\ocaM_{0,n})$ is the coefficient of $\Lbb^\ell$ 
in~\eqref{eq:M0ninp}. 
The logarithms of the individual factors in each summand in~\eqref{eq:M0ninp}
expand as follows.

$\bullet$ $\log\left((1-\Lbb)^{\frac{1+ k}\Lbb- k}\right)$:
\begin{align*}
\left(\frac{1+ k}\Lbb- k\right) \log(1-\Lbb) 
&=
-\sum_{i\ge 0} (1+ k) \frac{\Lbb^i}{i+1}+\sum_{i\ge 0}  k \frac{\Lbb^{i+1}}{i+1} \\
&=
-(1+ k)+\sum_{i\ge 1} \frac{ k-i}{i(i+1)} \Lbb^i
\saf.
\end{align*}

$\bullet$ $\log\left((1+\Lbb)^{\frac {1+ k}\Lbb -  k-n+1}\right)$:
\begin{align*}
\left(\frac {1+ k}\Lbb -  k-n+1\right) \log(1+\Lbb) 
&=(1+ k)\sum_{i\ge 0}(-1)^i\frac{\Lbb^i}{i+1}
-( k+n-1) \sum_{i\ge 0}(-1)^i\frac{\Lbb^{i+1}}{i+1} \\
&=
(1+ k)+\sum_{i\ge 1} (-1)^i\frac{2 k i+n i+ k+n-1}{i(i+1)} \Lbb^i 
\saf.
\end{align*}

$\bullet$ $\log\left(\prod_{j=0}^{ k+n-2} \left(1-\frac{j\Lbb}{ k+1}\right)\right)$:
\[
-\sum_{j=0}^{ k+n-2}\sum_{i\ge 0} \frac{j^{i+1}\Lbb^{i+1}}{(i+1)( k+1)^{i+1}}
=-\sum_{i\ge 1} \left(\sum_{j=0}^{ k+n-2} j^i\right) 
\frac{\Lbb^i}{i( k+1)^i}\saf.
\]
Therefore, the product of the three factors is the exponential of
$\sum_{i\ge 1} C_{n k i} \Lbb^i$
where for $n\ge 0$, $ k\ge 0$, $i\ge 1$ we set
\begin{equation}\label{eq:Cnlidef}
C_{n k i}:=\frac{(-1)^i (2 k i+ni+ k +n-1)+ k -i}{i(i+1)}
-\frac 1{i( k +1)^i}\sum_{j=0}^{ k +n-2} j^i\saf.
\end{equation}
With this notation,
\begin{align*}
[\ocaM_{0,n}] &=\sum_{\ell\ge 0} \frac{(\ell+1)^{\ell+n-1}}{(\ell+1)!}
\exp\left(\sum_{i\ge 1} C_{n\ell i} \Lbb^i\right) \Lbb^\ell \\
&=\sum_{\ell\ge 0} \frac{(\ell+1)^{\ell+n-1}}{(\ell+1)!}
\sum_{m\ge 0} \frac 1{m!} \left(\sum_{i\ge 1} C_{n\ell i} \Lbb^i\right)^m \Lbb^\ell
\saf.
\end{align*}
We have
\[
\left(\sum_{i\ge 1} C_{nk i} \Lbb^i\right)^m
=\sum_{r\ge 0} \left(\sum_{i_1+\cdots + i_m=r} C_{nki_1}\cdots C_{nki_m}\right)\Lbb^r
\]
where the indices $i_j$ in the summation are positive integers. The stated 
formula~\eqref{eq:betti}
\[
\rk H^{2\ell}(\ocaM_{0,n})
=\sum_{k=0}^\ell \frac{(k+1)^{k+n-1}}{(k+1)!}
\sum_{m=0}^{\ell-k} \frac 1{m!} 
\sum_{i_1+\cdots+i_m=\ell-k} C_{nki_1}\cdots C_{nki_m}
\]
follows.
\end{proof}

\begin{remark}
In parsing this formula, it is important to keep in mind that the indices $i_j$ in the
summation are positive. For example, when $k=\ell$, i.e., $\ell-k=0$, there is
exactly one contribution to the last $\sum$, that is, the empty choice of indices ($m=0$);
that summand is the empty product of the coefficients $C_{nki}$, that is, $1$. 
For instance, $\rk H^0(\ocaM_{0,n})=1$ for all~$n$.

By contrast, if $m>\ell-k$, then the last $\sum$ is the empty sum, so the contribution 
of such terms is~$0$. This is why we can bound $m$ by $\ell-k$ in the range of
summation.
\qede\end{remark}

The classical {\em Faulhaber's formula\/} states that for $i\ge 1$
\[
\sum _{j=0}^{N}j^{i}={\frac {1}{i+1}}\sum _{j=0}^{i}{\binom {i+1}{j}}B_{j}N^{i-j+1}
\]
where $B_j$ denote the Bernoulli numbers (with the convention $B_1=\frac 12$).
Applying Faulhaber's formula turns~\eqref{eq:Cnlidef} into
\begin{multline*}
C_{nk i}=\frac 1{i(i+1)}\bigg(
(-1)^i(2ki+ni+k+n-1)+k-i \\
-\frac 1{(k+1)^i}
\sum_{j=0}^i \binom{i+1}j B_j(k+n-1)^{i-j+1}
\bigg)\saf.
\end{multline*}
as stated in the introduction.

\begin{proof}[Proof of Theorem~\ref{thm:M0nbarGC}]
To prove Theorem~\ref{thm:M0nbarGC}, we use the second expression for $\hM$ obtained
in Theorem~\ref{thm:main}:
\[
\hM=\sum_{\ell\ge 0}\sum_{k\ge 0}\frac{(\ell+1)^{\ell+k}}{(\ell+1)! k!}
(z-\Lbb-z\Lbb)^k \prod_{j=0}^{\ell+k-1} \left(1-\frac{j\Lbb}{\ell+1}\right) \Lbb^\ell\saf.
\] 
The class is obtained by setting $z=0$ in the $(n-1)$-st derivative with respect to $z$.
This straightforward operation yields
\[
(1-\Lbb)^{n-1}\sum_{\ell\ge 0}\sum_{k\ge n-1}\frac{(\ell+1)^{\ell+k}}{(\ell+1)! (k-n+1)!}
(-\Lbb)^{k-n+1} \prod_{j=0}^{\ell+k-1} \left(1-\frac{j\Lbb}{\ell+1}\right) \Lbb^\ell
\]
and therefore, after simple manipulations,
\[
[\ocaM_{0,n}]=(1-\Lbb)^{n-1}
\sum_{k\ge 0}\sum_{\ell= 0}^k \frac{(-1)^{k-\ell} (\ell+1)^{k+n-1}}{(\ell+1)! (k-\ell)!}
\prod_{j=0}^{k+n-2} \left(1-\frac{j\Lbb}{\ell+1}\right) \Lbb^k\saf.
\]
Now recall that the Stirling numbers of the first kind, $s(N,j)$, are defined by the identity
\[
\sum_{j=0}^N s(N,j) x^j = \prod_{j=0}^{N-1} (x-j) = x (x-1)\cdots (x-N+1)\saf.
\]
Setting $y=\frac 1x$ and clearing denominators shows that
\[
\prod_{j=0}^{N-1} (1-jy) = \sum_{j=0}^N s(N,N-j) y^j\saf.
\]
We then have
\begin{align*}
\prod_{j=0}^{k+n-2} \left(1-\frac{j\Lbb}{\ell+1}\right) 
&= \sum_{j=0}^{k+n-1} s(k+n-1,k+n-1-j) \left(\frac{\Lbb}{\ell+1}\right)^j
\end{align*}
yielding
\[
[\ocaM_{0,n}]=(1-\Lbb)^{n-1}
\sum_{k\ge 0}\sum_{j=0}^{k+n-1} \sum_{\ell= 0}^k 
\frac{(-1)^{k-\ell} (\ell+1)^{k+n-1-j} }{(\ell+1)! (k-\ell)!} s(k+n-1,k+n-1-j)\Lbb^{k+j}\saf.
\]
Next, recall that the Stirling numbers of the second kind, $S(N,r)$, are defined by
\begin{equation}\label{eq:Stir2def}
S(N,r)=\sum_{i=1}^r (-1)^{r-i} \frac{i^N}{i!(r-i)!}
\end{equation}
for all $N>0$, $r\ge 0$.
Therefore
\[
\sum_{\ell= 0}^k (-1)^{k-\ell} \frac{(\ell+1)^{k+n-1-j} }{(\ell+1)! (k-\ell)!} 
= \sum_{i=1}^{k+1} (-1)^{k+1-i} \frac{i^{k+n-1-j} }{i! (k+1-i)!} 
=S(k+n-1-j,k+1)
\]
and we can conclude
\[
[\ocaM_{0,n}]=(1-\Lbb)^{n-1}
\sum_{k\ge 0}\sum_{j=0}^{k+n-1} S(k+n-1-j,k+1) s(k+n-1,k+n-1-j)\Lbb^{k+j}\saf,
\]
which is the statement.
\end{proof}

\begin{remark}
Since $S(N,r)=0$ if $r>N$, nonzero summands in this expression only occur for $0\le j\le n-2$.
\qede\end{remark}

Corollary~\ref{cor:bettisim} follows immediately from Theorem~\ref{thm:M0nbarGC}, as
the reader may verify.

%%%

\section{The generating function $P$, and proof of Theorem~\ref{thm:main2}}\label{sec:polp}

We now move to the polynomials $p^{(k)}_{m}(z)\in \Qbb[z]$ mentioned in~\S\ref{sec:intro}. 
These polynomials were introduced in~\cite{ACM} for the purpose of studying the
generating functions for individual Betti numbers of $\ocaM_{0,n}$. As proved in~\cite{ACM}, for all $m\ge 0$ and $0\le m\le k$ there exist polynomials $p_m^{(k)}(z)\in \Qbb[z]$ 
of degree $2m$ such that
\begin{equation}\label{eq:Hkfromp}
\sum_{n\ge 3} \rk H^{2k}(\ocaM_{0,n})\frac{z^{n-1}}{(n-1)!} 
= e^z\sum_{m=0}^k (-1)^m p_m^{(k)}(z)\, e^{(k-m) z}
\end{equation}
for all $k\ge 0$. 
Let
\[
P(z,t,u):=\sum_{\ell\ge 0}\sum_{m\ge 0} p^{(m+\ell)}_m(z) u^m t^\ell
\]
be the generating function for the polynomials $p^{(k)}_m$; then \eqref{eq:Hkfromp} 
states that
\begin{equation}\label{eq:MfromP}
\hM(z,\Lbb) = e^z P(z,e^z\Lbb,-\Lbb)\saf.
\end{equation}
In this section we prove Theorem~\ref{thm:main2} from the introduction, which 
gives an explicit expression for the generating function $P$:
\[
P(z,t,u) =\sum_{\ell\ge 0} \left(
\frac{(\ell+1)^\ell}{(\ell+1)!} e^{-(\ell+1)z}\,
\left((1+u)(1-u(z+1))\right)^{-\frac{1+\ell(u+1)}u}\,
\prod_{j=0}^{\ell-1} \left(1+\frac {ju}{\ell+1}\right)\right) t^\ell\saf.
\]
This statement will follow from Theorem~\ref{thm:main}, but the argument is not
completely trivial for the mundane reason that~$\hM(z,\Lbb)$ is a two-variable function 
while $P(z,t,u)$ is a three-variable function; $\hM(z,\Lbb)$~is a specialization of $P(z,t,u)$, 
not conversely.

\begin{proof}[Proof of~Theorem~\ref{thm:main2}]
We let
\[
\oP(z,t,u) =\sum_{\ell\ge 0} \left(
\frac{(\ell+1)^\ell}{(\ell+1)!} e^{-(\ell+1)z}\,
\left((1+u)(1-u(z+1))\right)^{-\frac{1+\ell(u+1)}u}\,
\prod_{j=0}^{\ell-1} \left(1+\frac {ju}{\ell+1}\right)\right) t^\ell
\]
and we have to prove that $\oP(z,t,u)=P(z,t,u)$. Equivalently, we will prove that
\begin{equation}\label{eq:oPoP}
\oP(z,su,-u)=P(z,su,-u)\saf;
\end{equation}
note that the right-hand side equals 
$\sum_{k\ge 0} \sum_{m=0}^k (-1)^m p^{(k)}_m s^{k-m} u^k$.

\begin{claim}\label{claim:pkzspol}
There exist {\em polynomials\/} $\overline p^{(k)}(z,s)\in \Qbb[z,s]$ such that
\[
\oP(z,su,-u)=\sum_{k\ge 0} \overline p^{(k)}(z,s) u^k\saf.
\]
\end{claim}

This claim implies~\eqref{eq:oPoP}. Indeed, 
by Theorem~\ref{thm:main} and identity~\eqref{eq:Hkfromp} we have 
\[
\sum_{k\ge 0} \overline p^{(k)}(z,e^z) \Lbb^k =
\oP(z,e^z \Lbb,-\Lbb) = e^{-z} \hM(z,\Lbb)=
\sum_{k\ge 0} \sum_{m=0}^k (-1)^m  p^{(k)}_m(z) e^{(k-m) z} \Lbb^k
\]
and therefore
\begin{equation}\label{eq:opp}
\overline p^{(k)}(z,e^z) = \sum_{m=0}^k (-1)^m  p^{(k)}_m(z) e^{(k-m) z}
\end{equation}
for all $k\ge 0$. Given that Claim~\ref{claim:pkzspol} holds,
\[
\overline p^{(k)}(z,s) - \sum_{m=0}^k (-1)^m  p^{(k)}_m(z) s^{(k-m)}
\]
is then a {\em polynomial\/} vanishing at $s=e^z$, hence it must be identically $0$ 
since $e^z$ is transcendental over $\Qbb(z)$. The needed identity~\eqref{eq:oPoP} follows.

Thus, we are reduced to verifying~Claim~\ref{claim:pkzspol}.

By definition,
\[
\oP(z,su,-u)=\sum_{\ell\ge 0} \left(
\frac{(\ell+1)^\ell}{(\ell+1)!} e^{-(\ell+1)z}\,
\left((1-u)(1+u(z+1))\right)^{\frac{1+\ell(1-u)}u}\,
\prod_{j=0}^{\ell-1} \left(1-\frac {ju}{\ell+1}\right)\right) s^\ell u^\ell\saf.
\]
In order to prove~Claim~\ref{claim:pkzspol}, we have to verify that for all $k\ge 0$
the coefficient of $u^k$ in
\[
\sum_{\ell= 0}^k \left(
\frac{(\ell+1)^\ell}{(\ell+1)!} e^{-(\ell+1)z}\,
\left((1-u)(1+u(z+1))\right)^{\frac{1+\ell(1-u)}u}\,
\prod_{j=0}^{\ell-1} \left(1-\frac {ju}{\ell+1}\right)\right) s^\ell u^\ell
\]
is a polynomial in $z$ and $s$.
This is clearly a polynomial in $s$, and it suffices then to verify that the coefficient 
of $u^i$ in the factor
\[
e^{-(\ell+1)z} (1+u(z+1))^{\frac {1+\ell(1-u)}u}
\]
is a polynomial in $z$ for all $i\ge 0$. Simple manipulations show that this factor equals
\begin{multline*}
\exp \left(\frac{1+\ell(1-u)}u\log(1+u(z+1)) -(1+\ell)z\right) \\
=\sum_{k\ge 0} \frac 1{k!}\left((1+\ell)+\sum_{i\ge 0} \left(\frac{(1+\ell)(z+1)}{i+2}
+  \frac{\ell}{i+1}\right)(-u(z+1))^{i+1} \right)^k
\end{multline*}
and the statement is clear from this expression. (In fact, the coefficient of $u^i$ in
this expression is a polynomial of degree $2i$ in $z$.)
This concludes the proof of Claim~\ref{claim:pkzspol}, and therefore of 
Theorem~\ref{thm:main2}.
\end{proof}

Of course the same result may be formulated in different ways. The following 
expression follows from Theorem~\ref{thm:main2} by manipulations analogous to
the corresponding manipulations in the proof of Theorem~\ref{thm:main}.

\begin{corol}
\[
P(z,t,u) = \sum_{\ell\ge 0}\sum_{k\ge 0}\left(
\frac{(\ell+1)^{\ell+k}}{(\ell+1)! k!} (z+u+zu)^k 
\prod_{j=0}^{\ell+k-1} \left(1 + \frac{ju}{\ell+1}\right)\right)
e^{-(\ell+1)z} t^\ell\saf.
\]
\end{corol}

\begin{remark}
It is straightforward to verify that the function $P(z,t,u)$ is a solution of the differential equation
\[
\frac{\partial P}{\partial z} + t \frac{\partial P}{\partial t} +P
=\frac P{1-u(1+z)-t  P}\saf.
\]
This is the analogue for $P$ of the differential equation~\eqref{eq:Mdifeq} satisfied
by $\hM$.

It is also not difficult to obtain a functional equation for $P$,
\[
(e^z P)^{-u} + u t P = (1+u)(1-u(z+1))\saf,
\]
lifting the functional equation~\eqref{eq:funeqM} satisfied by $\hM$.
\qede\end{remark}

%%%

\section{The polynomials $p^{(k)}_m$ and $\Gamma_{mj}$, and proof of 
Theorems~\ref{thm:pcG} and~\ref{thm:betti2}}\label{sec:pkm}

We are interested in studying more thoroughly the polynomials $p^{(k)}_m(z)\in \Qbb[z]$ 
defined in~\S\ref{sec:polp}. As proved in~\cite{ACM}, $p^{(k)}_m$ has degree~$2m$ and 
positive leading coefficient, and these polynomials determine the Betti 
numbers of $\ocaM_{0,n}$ in the sense that 
\[
\sum_{n\ge 3} \rk H^{2k}(\ocaM_{0,n})\frac{z^{n-1}}{(n-1)!} 
= e^z\sum_{m=0}^k (-1)^m p_m^{(k)}(z)\, e^{(k-m) z}
\]
(cf.~\eqref{eq:Hkfromp}).
For instance, it follows that for every $k\ge 0$ the sequence of Betti numbers 
$\rk H^{2k}(\ocaM_{0,n})$ as $n=3,4,5,\dots$ is determined by a finite amount of information, 
specifically, the $(k+1)^2$ coefficients of the polynomials $p^{(k)}_m$, $m=0,\dots,k$. 
It is hoped that more information about the polynomials 
$p^{(k)}_m$ will help in proving conjectured facts about the integers $\rk H^{2k}(\ocaM_{0,n})$, 
such as log-concavity properties.

The following list of the first several polynomials $p^{(k)}_m$ is reproduced 
from~\cite{ACM}.
\begin{align*}
p_0^{(0)} &=1 \\[10pt]
p_0^{(1)} &=1 \\
p_1^{(1)} &=\frac 12 z^2 + z + 1 \\[10pt]
p_0^{(2)} &=\frac 32 \\
p_1^{(2)} &=z^2 + 3z + 2 \\
p_2^{(2)} &=\frac 18 z^4 + \frac 56 z^3 + 2 z^2 + 2 z + \frac 12 \\[10pt]
p_0^{(3)} &=\frac 83 \\
p_1^{(3)} &=\frac 94 z^2 + \frac{15}2 z + 5 \\
p_2^{(3)} &=\frac 12 z^4 + \frac{11}3 z^3 + 9 z^2 + 9z + 3 \\
p_3^{(3)} &=\frac 1{48} z^6 + \frac 7{24} z^5 + \frac{35}{24} z^4 + \frac 72 z^3 
+ \frac{17}4 z^2 + \frac 52 z + \frac 23 \\[10pt]
p_0^{(4)} &=\frac {125}{24} \\
p_1^{(4)} &=\frac{16}3 z^2 + \frac{56}3 z + \frac{38}3 \\
p_2^{(4)} &=\frac{27}{16} z^4 + \frac{51}4 z^3 + \frac{129}4 z^2 + \frac{65}2 z 
+ \frac{45}4 \\
p_3^{(4)} &=\frac 16 z^6 + \frac{13}6 z^5 + \frac{21}2 z^4 + \frac{74}3 z^3 
+ 30 z^2 + 18 z + \frac{13}3 \\
p_4^{(4)} &=\frac 1{384} z^8 + \frac 1{16} z^7 + \frac 59 z^6 + \frac {49}{20} z^5 
+ \frac{289}{48} z^4 + \frac{103}{12} z^3 + \frac {85}{12} z^2 + \frac {19}6 z + \frac{13}{24} 
\end{align*}
In this section we prove that all polynomials $p^{(k)}_m$ have positive coefficients and 
provide evidence for the assertion that most $p^{(k)}_m$ are ultra-log-concave. Both
facts were conjectured in~\cite{ACM}.

We denote by $c^{(k)}_{mj}\in \Qbb$ the coefficient of $z^j$ in $p^{(k)}_m(z)$.
The following proposition introduces rational numbers $\Gamma_{mj}(\ell)$ and 
establishes the equality~\eqref{eq:cfromGamma} stated in~Theorem~\ref{thm:pcG}.

\begin{prop}\label{prop:Gac}
For $\ell\ge 0$ and $j\ge 0$ let $\Gamma_{mj}(\ell)\in \Qbb$ be defined by 
the identity
\[
\sum_{m\ge 0} \sum_{j=0}^{2m} \Gamma_{mj}(\ell) z^j u^m
=e^{-z}\left((1+(\ell+1)u)(1-(z+\ell+1)u)\right)^{-\frac 1u-\ell} \prod_{j=0}^{\ell-1}
(1+ju)\saf.
\]
Then $\Gamma_{mj}(\ell)=0$ for $j>2m$ and
\[
c^{(k)}_{mj} =\frac{(k-m+1)^{k-2m+j}}{(k-m+1)!} \cdot \Gamma_{mj}(k-m)
\]
for all $k$, $m$, $j$ with $0\le m\le k$, $0\le j\le 2m$.
\end{prop}

\begin{proof}
This follows from an application of Theorem~\ref{thm:main2}:
\begin{align*}
&\sum_{\ell\ge 0} \sum_{m\ge 0} \sum_{j\ge 0}\frac{(\ell+1)^{\ell-m+j}}{(\ell+1)!}
\Gamma_{mj}(\ell) z^j u^m t^\ell \\
=& \sum_{\ell\ge 0} \frac{(\ell+1)^\ell}{(\ell+1)!} \sum_{m\ge 0} \sum_{j\ge 0}
\Gamma_{mj}(\ell) ((\ell+1)z)^j \left(\frac u{\ell+1}\right)^m t^\ell \\
=&\sum_{\ell\ge 0} \frac{(\ell+1)^\ell}{(\ell+1)!} e^{(\ell+1)z}
\left( \left(1+u\right)\left(1-(z+1)u\right)\right)^{-\frac{\ell+1}u-\ell} 
\prod_{j=0}^{\ell-1}\left(1+\frac{ju}{\ell+1}\right) \\
=&P(z,t,u) 
=\sum_{\ell\ge 0} \sum_{m\ge 0} p^{(m+\ell)}(z) u^m t^\ell \\
=&\sum_{\ell\ge 0} \sum_{m\ge 0} \sum_{j=0}^{2m} c^{(m+\ell)}_{mj} z^j u^m t^\ell\saf.
\end{align*}
Comparing coefficients gives the statement.
\end{proof}

In order to prove Theorem~\ref{thm:pcG} we have to verify that $\Gamma_{mj}(\ell)$
is a polynomial in $\ell$ and attains positive values for $\ell\ge 0$ in the range
$m\ge 0$, $0\le j\le 2m$. For this purpose, we introduce another set of ancillary 
rational numbers $\Delta_{mj}(\ell)$, defined by the generating function
\begin{equation}\label{eq:Deldef}
\sum_{m\ge 0} \sum_{j\ge 0} \Delta_{mj}(\ell)z^j u^m
=e^{-z}\,\left((1+(\ell+1)u)(1-u(z+\ell+1))\right)^{-(\frac 1u+\ell)}\saf;
\end{equation}
thus,
\begin{equation}\label{eq:GaDe}
\sum_{m\ge 0} \sum_{j=0}^{2m} \Gamma_{mj}(\ell)z^j u^m
=\left(\sum_{m\ge 0} \sum_{j=0}^{2m} \Delta_{mj}(\ell)z^j u^m\right)
\prod_{j=0}^{\ell-1} \left(1+ju\right)\saf.
\end{equation}

\begin{prop}\label{prop:Delta}
For all $m\ge 0$ and $j=0,\dots,2m$, 
$\Delta_{mj}(\ell)$ is a polynomial in $\ell$ of degree $2m-j$ and with positive 
coefficients. For $j>2m$, $\Delta_{mj}(\ell)=0$. 
\end{prop}

\begin{proof}
The proof amounts to elementary but delicate calculations. Note that
\begin{align*}
\log &\left(e^{-z}\,\left((1+(\ell+1)u)(1-u(z+\ell+1))\right)^{-(\frac 1u+\ell)}\right) \\
&=-z+\left(-\frac 1u-\ell\right)
\left(\log(1+(\ell+1)u)+\log(1-(\ell+1+z)u)\right) \\
&=-z+\left(-\frac 1u-\ell\right)
\sum_{i\ge 0} \frac{(-1)^i (\ell+1)^{i+1}-(\ell+1+z)^{i+1}}{i+1} u^{i+1} \saf.
\end{align*}
The reader will verify that this expression equals
\begin{multline*}
\sum_{i\ge 1} 
\frac 1{i(i+1)}\Big(
(\ell+1)^i(i+2\ell i+\ell+ (-1)^{i+1}(i-\ell)) \\
+\sum_{j=1}^i \binom {i+1}j (\ell+1)^{i-j} 
(\ell(i-j)+i+\ell +\ell i)z^j+ i z^{i+1}\Big)
u^i \saf,
\end{multline*}
that is, the coefficient of $z^j u^i$ in the expression is $0$ for $i=0$ and for
$j>i+1$, and equals
\[
\begin{cases}
\frac 1{i(i+1)}(\ell+1)^i(i+2\ell i+\ell+ (-1)^{i+1}(i-\ell))\quad &\text{for $j=0$} \\
\frac 1{i(i+1)}\binom {i+1}j (\ell+1)^{i-j} (\ell(i-j)+i+\ell +\ell i)\quad &\text{for $1\le j\le i$} \\
\frac 1{i+1} &\text{for $j=i+1$} 
\end{cases}
\]
for $i\ge 1$.
For all $j=0,\dots, i+1$, this is a polynomial in $\ell$ of degree $i+1-j$.
Further, trivially $(i-j)\ge 0$ in the range $1\le j\le i$ and
\[
i+2\ell i+\ell+ (-1)^{i+1}(i-\ell) =
\begin{cases}
2i(\ell+1)> 0\quad & \text{if $i$ is odd} \\
2\ell (i+1)> 0\quad & \text{if $i$ is even.} 
\end{cases}
\]
Summarizing,
\[
\sum_{m\ge 0} \sum_{j\ge 0} \Delta_{mj}(\ell)z^j u^m
=\exp\left(\sum_{i\ge 1} \sum_{j=0}^{i+1} \delta_{ij}(\ell) z^j u^i\right)
\]
with $\delta_{ij}(\ell)$ polynomials of degree $i+1-j$ with positive coefficients. Now
\[
\exp(c_1 u + c_2 u^2 + c_3 u^3 + \cdots)=1 + c_1 u + \left(\frac {c_1^2}2  + c_2\right)u^2 
+ \left(c_3 + c_1 c_2 + \frac {c_1^3}6\right)u^3 + \cdots
\]
is a series whose coefficient of $u^m$ is a linear combination of products
$c_{i_1}\cdots c_{i_r}$
with $i_k>0$ and $\sum i_k=m$. Therefore: for every $m$, 
$\sum_{j\ge 0} \Delta_{mj}(\ell) z^j$ is a linear combination of terms
\[
\left( \sum_{j_1=0}^{i_1+1} \delta_{i_1j_1}(\ell) z^{j_1} \right)\cdots
\left( \sum_{j_r=0}^{i_r+1} \delta_{i_rj_r}(\ell) z^{j_r} \right)
=\sum \delta_{i_1j_1}(\ell)\cdots \delta_{i_rj_r}(\ell) z^{j_1+\cdots +j_r}\saf.
\]
with $i_k>0$ and $\sum i_k=m$. Each polynomial 
$\delta_{i_1j_1}(\ell)\cdots \delta_{i_rj_r}(\ell)$ has positive coefficients and degree
\[
\sum (i_k+1)-\sum {j_k}= m+r-j
\]
with $j=\sum j_k$ the exponent of $z$. The degree of the sum is the maximum
attained by $m+r-j$, that is, $2m-j$ (for $r=m$, $i_1=\cdots=i_m=1$). 
Thus, the coefficient $\Delta_{mj}(\ell)$ has degree $2m-j$ in $\ell$, as stated. 
The exponent $j$ itself ranges from $0$, attained for $j_r=0$ for all $r$, to 
$\sum (i_k+1)=m+r$, and attains the maximum~$2m$, again when all $i_k$ 
equal~$1$. 

Thus $\Delta_{mj}(\ell)=0$ for $j>2m$, and this concludes the proof of the proposition.
\end{proof}

The polynomials $\Delta_{mj}(\ell)$ can of course be computed explicitly:
\[
\Delta_{00}=1,\quad
\Delta_{10}= (\ell+1)^2,\quad
\Delta_{11}= 2\ell+1,\quad
\Delta_{12}= \frac 12,\quad
\Delta_{20}= \frac{(\ell^2 + 4\ell + 1)(\ell + 1)^2}2,
\]
etc. We have verified that all the polynomials $\Delta_{mj}(\ell)$ with $1\le m\le 50$,
$0\le j\le 2m$, are ultra-log-concave. The first several hundred are in fact real-rooted,
but $\Delta_{19,1}(\ell)$ appears not to be.

Proposition~\ref{prop:Delta} implies another part of Theorem~\ref{thm:pcG}, thereby
confirming the first statement in Conjecture~\ref{conj:plogc}.

\begin{corol}\label{cor:pospkm}
For all $m\ge 0$, $0\le j\le 2m$, and $\ell\ge 0$ we have $\Gamma_{mj}(\ell)>0$.

Therefore, for all $m\ge 0$, $0\le j\le 2m$, and $k\ge m$, $p^{(k)}_m(z)$ is a degree~$2m$ 
polynomial with {\em positive\/} coefficients.
\end{corol}

\begin{proof}
The second part of the statement follows from the first by Proposition~\ref{prop:Gac}. 

The first part is a consequence of~\eqref{eq:GaDe}, since the polynomials~$\Delta_{mj}(\ell)$ 
have positive coefficients by Proposition~\ref{prop:Delta} and so does the factor 
$\prod_{j=0}^{\ell-1} (1+ju)$.
\end{proof}

The following proposition will complete the proof of Theorem~\ref{thm:pcG}.

\begin{prop}\label{prop:Gampo}
For all $m\ge 0$ and $0\le j\le m$, the functions $\Gamma_{mj}(\ell)$ are polynomials 
of degree $2m-j$ in $\ell$.
\end{prop}

\begin{proof}
Faulhaber's formula for sums of powers $\sum_{j=0}^{\ell-1} j^i$ easily implies that
\begin{align*}
\prod_{j=0}^{\ell-1} &(1+ju) =\exp\left(
\sum_{i\ge 1} 
\sum_{j=0}^i \binom{i+1}j B_j \ell^{i-j+1}
\frac{(-1)^{i+1}u^i}{i(i+1)}
\right) \\
&=1+\frac 12 \ell(\ell-1)u + \frac 1{24} \ell(\ell-1)(\ell-2) (3\ell-1) u^2
+\frac 1{48} \ell^2 (\ell-1)^2(\ell-2)(\ell-3)u^3 + \cdots \\
&=: \sum_{m\ge 0} \beta_m(\ell) u^m
\end{align*}
where $B_j$ denotes the $j$-th Bernoulli number. It follows that the coefficient $\beta_m(\ell)$
of $u^m$ in the expansion of $\prod_{j=0}^{\ell-1}(1+ju)$ is a polynomial in $\Qbb[\ell]$
of degree~$2m$ (and of course $\beta_m(\ell)=0$ for $m\ge \ell$). 
By~\eqref{eq:GaDe} we have
\[
\Gamma_{mj}(\ell)=\sum_{m_1+m_2=m} \Delta_{m_1j}(\ell) \beta_{m_2}(\ell)
\]
and it follows that $\Gamma_{mj}(\ell)$ is a polynomial in $\Qbb[\ell]$ of degree
$2m-j$ as stated.
\end{proof}

This concludes the proof of Theorem~\ref{thm:pcG}.\smallskip

The information collected so far may also be used to investigate the conjectured log-concavity
properties of the polynomials $p^{(k)}_m$ (cf.~Conjecture~\ref{conj:plogc} in~\S\ref{sec:intro}).
With the notation introduced above, ultra-log concavity for $p^{(k)}_m$ is the following condition:
\[
\forall j = 1,\dots,2m-1\colon\quad
\left(\frac{c^{(k)}_{mj}}{\binom {2m}j}\right)^2 \ge \frac{c^{(k)}_{m,j-1}}{\binom {2m}{j-1}}\cdot 
\frac{c^{(k)}_{m,j+1}}{\binom {2m}{j+1}}\saf.
\]
By Corollary~\ref{cor:pospkm} $c^{(k)}_{mj}>0$ in this range, therefore these sequences
automatically have no internal zeros.

\begin{lemma}\label{lem:testlcp}
The polynomial $p^{(m+\ell)}_m$ is ultra-log-concave if and only if
\[
\forall j=1,\dots,2m-1\colon\quad
j(2m-j) \Gamma_{mj}(\ell)^2 - (j+1)(2m-j+1) \Gamma_{m,j-1}(\ell) \Gamma_{m,j+1}(\ell)\ge 0
\saf.
\]
\end{lemma}

\begin{proof}
Immediate consequence of Proposition~\ref{prop:Gac}.
\end{proof}

Since by Proposition~\ref{prop:Gampo} each $\Gamma_{mj}(\ell)$ is a {\em polynomial\/}
in $\ell$, Lemma~\ref{lem:testlcp} provides us with an effective way to test the 
ultra-log-concavity of all polynomials
$p^{(k)}_m$, $k\ge m$, for any given~$m$. For example, according to Lemma~\ref{lem:testlcp},
in order to verify that $p^{(k)}_2$ is ultra-log-concave for all $k\ge 2$, it suffices to observe
that the polynomials
\begin{gather*}
\frac 94 \ell^6+\frac{131}{12} \ell^5+\frac{41}3 \ell^4 + \frac{85}4 \ell^3
+\frac{403}{12} \ell^2 +\frac{67}3 \ell +4 \\
\frac{13}4 \ell^4 +\frac{25}2 \ell^3 + 9 \ell^2 +\frac 54\ell+1 \\
\frac 14 \ell^2 + \frac 34 \ell + \frac 1{12}
\end{gather*}
are trivially positive-valued for all $\ell\ge 0$. For larger $m$ not all the polynomials appearing 
in Lemma~\ref{lem:testlcp} have positive coefficients, but it is a straightforward computational
process to verify whether they only attain positive values for $\ell\ge 0$.

\begin{prop}
The polynomials $p^{(k)}_m(z)$ are ultra-log concave for $m\le 100$ and all $k\ge m$, 
with the exceptions listed in Conjecture~\ref{conj:plogc}.
\end{prop}

\begin{proof}
As in the foregoing discussion, Lemma~\ref{lem:testlcp} reduces this statement to a finite 
computation, which we carried out with maple.
\end{proof}

As one more application of the material developed above, we have the following 
formula for the Betti numbers of $\ocaM_{0,n}$ in terms of the polynomials $\Gamma_{mj}$:
\begin{equation}\label{eq:betti2ag}
\rk H^{2\ell}(\ocaM_{0,n}) =
\sum_{k+m=\ell} 
(-1)^m \frac{(k+1)^{n-2+k-m}}{k!}
\sum_{j=0}^{2m}
(n-1)\cdots (n-j)\,
\Gamma_{mj}(k)\saf.
\end{equation}
(This is Theorem~\ref{thm:betti2}.)

\begin{proof}
According to~\cite[Theorem~5.1]{ACM}, 
\[
\rk H^{2\ell}(\ocaM_{0,n})=\frac{(\ell+1)^{\ell+n-1}}{(\ell+1)!} 
+\sum_{m=1}^\ell (-1)^m\sum_{j=0}^{2m} \binom{n-1}j\, c_{mj}^{(\ell)}\, j!
(\ell-m+1)^{n-1-j}\saf.
\]
Identity~\eqref{eq:betti2ag} follows then from Proposition~\ref{prop:Gac}.
\end{proof}

%%%

\section{The Lambert W function and proof of Theorem~\ref{thm:Mtree}}\label{sec:Wm1}

The goal of this section is the proof of Theorem~\ref{thm:Mtree}, expressing the
generating function $\hM=1+z+\sum_{n\ge 3} [\ocaM_{0,n}]\frac{z^{n-1}}{(n-1)!}$ 
in terms of the principal branch $W$ of the classical {\em Lambert~W function,\/} 
cf.~\cite{MR1414285}. In fact it is more notationally convenient to work with the 
`tree function', defined by
\[
T(t) = -W(-t) =\sum_{n\ge 1}\frac{n^{n-1} t^n}{n!}
\]
and we recall that this function satisfies the relation
\begin{equation}\label{eq:Trel}
T(t)=t e^{T(t)}
\end{equation}
(\cite[(1.5)]{MR1414285}).

Consider the coefficient of $u^m$ in the generating function $P$ introduced in~\S\ref{sec:polp}:
\[
P_m(z,t)=\sum_{\ell\ge 0} p^{(m+\ell)}_m(z) t^\ell
\]
(cf.~\eqref{eq:gfPm}). It is clear that this function may be expressed as a series
in the tree function $T=T(t)$, since $t=e^{-T} T$ according to~\eqref{eq:Trel}:
\begin{equation}\label{eq:PmT}
P_m(z,t)=\sum_{\ell\ge 0} p^{(m+\ell)}_m(z) e^{-\ell T(t)} T(t)^\ell\saf.
\end{equation}
We are going to verify that this series is a {\em rational function\/} in $T$, and in fact
admits the particularly simple expression \eqref{eq:Pmtree}:
\[
P_m=e^T \frac{F_m(z,T)}{(1-T)^{2m-1}}
\]
where $F_0=\frac 1{1-T}$ and $F_m$ are polynomials for $m>0$, of degree $< 3m$. 
This is Proposition~\ref{prop:treeprop}, which we proceed to prove. Ultimately, this fact
is a consequence of Proposition~\ref{prop:Gampo}.

\begin{proof}[Proof of Proposition~\ref{prop:treeprop}]
For $m=0$, using~\eqref{eq:pk0}:
\[
P_0(z,t)=\sum_{\ell\ge 0} p^{(\ell)}_0 t^\ell=\sum_{\ell\ge 0} \frac{(\ell+1)^\ell}{(\ell+1)!} t^\ell
=\frac 1t T(t) = e^T
\]
by~\eqref{eq:Trel}, and this is the statement. For $m>0$, and treating $T$ as an indeterminate, 
consider the expression
\begin{align}\label{eq:bridge}
(1-T)^{2m-1} &\sum_{\ell\ge 0} 
\frac{(\ell+1)^{\ell+a}}{(\ell+1)!} e^{-(\ell+1)T} T^{\ell} \\
\notag &=(1-T)^{2m-1} \sum_{\ell\ge 0} \sum_{k\ge 0}
\frac{(\ell+1)^{\ell+k+a}}{(\ell+1)! k!} (-1)^k T^{k+\ell}\saf.
\end{align}
For $1\le a$, \eqref{eq:bridge} can be expressed in terms of Stirling numbers of 
the second kind (cf.~\eqref{eq:Stir2def}) and we can apply 
Lemma~\ref{lem:St2}:
\begin{align*}
(1-T)^{2m-1} &
\sum_{N\ge 0} \sum_{\ell\ge 0}\frac{(\ell+1)^{N+a}}{(\ell+1)! (N-\ell)!} (-1)^{N-\ell} T^N \\
&=(1-T)^{2m-1} 
\sum_{N\ge 0} \sum_{i\ge 1}\frac{i^{N+a}}{i!(N+1-i)!} (-1)^{N+1-i} T^N \\
&=(1-T)^{2m-1} \sum_{N\ge 0} S(N+a,N+1) T^N \\
&=(1-T)^{2(m-a)}\sigma_a(T)
\end{align*}
with $\sigma_a(T)$ a polynomial of degree $\le a-1$. In particular, for $a\le m$, 
\eqref{eq:bridge} is a polynomial of degree $\le 2m-a-1$ in this case. If $a\le 0$ and
$N\ge -a$, 
\[
\sum_{i\ge 1}\frac{i^{N+a}}{i!(N+1-i)!} (-1)^{N+1-i}
=S(N+a,N+1)=0
\]
since $N+a\le N+1$. Thus, \eqref{eq:bridge} is a polynomial of degree $\le 2m-a-1$ in 
this case as well. 

It follows that for all $m\ge 0$, $0\le j\le 2m$, and $0\le r\le 2m-j$,
\[
(1-T)^{2m-1} \sum_{\ell\ge 0} 
\frac{(\ell+1)^{\ell-m+j}}{(\ell+1)!} (\ell+1)^r e^{-(\ell+1)T} T^{\ell}
\]
is a polynomial of degree $\le 3m-j-r-1$. By Proposition~\ref{prop:Gampo}, 
$\Gamma_{mj}(\ell)$ is a polynomial of degree $2m-j$. Therefore it is a linear combination
of polynomials $(\ell+1)^r$, $0\le r\le 2m-j$, and we can conclude that
\[
(1-T)^{2m-1} \sum_{\ell\ge 0} 
\frac{(\ell+1)^{\ell-m+j}}{(\ell+1)!} \Gamma_{mj}(\ell) e^{-(\ell+1)T} T^{\ell}
\]
is a polynomial of degree $\le 3m-1$. By Proposition~\ref{prop:Gac}, this implies that
\[
F_m(z,T):=(1-T)^{2m-1} \sum_{\ell\ge 0} p_m^{(m+\ell)}(z) e^{-(\ell+1)T} T^{\ell}
\]
is a polynomial of degree $\le 3m-1$ in $T$. Now set $T=T(t)$ and use~\eqref{eq:PmT}:
\[
P_m(z,t)=\sum_{\ell\ge 0} p^{(m+\ell)}_m(z) e^{-\ell T(t)} T(t)^\ell
=e^{T(t)}\frac{ F_m(z,T(t))}{(1-T(t))^{2m-1}}
\]
to verify~\eqref{eq:Pmtree} as needed.
\end{proof}

Theorem~\ref{thm:Mtree} follows from Proposition~\ref{prop:treeprop}. Indeed, 
recall~\eqref{eq:MfromP}:
\[
\hM(z,\Lbb) = e^z P(z,e^z\Lbb,-\Lbb)\saf;
\]
since $P(z,t,u)=\sum_{m\ge 0} P_m(z,t) u^m$, Proposition~\ref{prop:treeprop} gives
\[
\hM(z,\Lbb) = e^z e^T\sum_{m\ge 0} \frac{ F_m(z,T)}{(1-T)^{2m-1}} (-\Lbb)^m
\]
where now $T=T(e^z\Lbb)$; by~\eqref{eq:Trel},
\[
e^z e^{T(e^z \Lbb)}= e^z e^{-z} \Lbb^{-1} T(e^z \Lbb)=\frac T\Lbb
\]
and~\eqref{eq:Mtree} follows.\qed

%%%

\section{Stirling matrices and the Grothendieck class of $\ocaM_{0,n}$}\label{sec:Stirma}

The material in this section is included mostly for aesthetic reasons.
We recast Theorem~\ref{thm:M0nbarGC} in terms of products involving 
the infinite matrices $\fs=(s(i,j))_{i,j\ge 1}$, resp., $\fS=(S(i,j))_{i,j\ge 1}$, defined by Stirling 
numbers of the first, resp., second kind:
\[
\fs = 
\begin{pmatrix}
1 & 0 & 0 & 0 & 0 & \cdots \\
-1 & 1 & 0 & 0 & 0 & \cdots \\
2 & -3 & 1 & 0 & 0 & \cdots \\
-6 & 11 & -6 & 1 & 0 & \cdots \\
24 & -50 & 35 & -10 & 1 & \cdots \\
\vdots & \vdots & \vdots & \vdots & \vdots & \ddots
\end{pmatrix},\quad
\fS = 
\begin{pmatrix}
1 & 0 & 0 & 0 & 0 & \cdots \\
1 & 1 & 0 & 0 & 0 & \cdots \\
1 & 3 & 1 & 0 & 0 & \cdots \\
1 & 7 & 6 & 1 & 0 & \cdots \\
1 & 15 & 25 & 10 & 1 & \cdots \\
\vdots & \vdots & \vdots & \vdots & \vdots & \ddots
\end{pmatrix}\saf.\quad
\]
We will also use the notation
\[
\one_\gamma:=\begin{pmatrix}
1 & 0 & 0 & 0 & \cdots \\
0 & \gamma & 0 & 0 & \cdots \\
0 & 0 & \gamma^2 & 0 & \cdots \\
0 & 0 & 0 & \gamma^3 & \cdots \\
\vdots & \vdots & \vdots & \vdots & \ddots
\end{pmatrix},\quad
\Sh_k:=\begin{pmatrix}
\overbrace{0 \quad \cdots \quad 0}^k & 1 & 0 & 0 & \cdots \\
0 \quad \cdots \quad 0 & 0 & 1 & 0 & \cdots \\
0 \quad \cdots \quad 0 & 0 & 0 & 1 & \cdots \\
\vdots \quad \vdots \quad \vdots & \vdots & \vdots & \vdots & \ddots
\end{pmatrix}\saf.
\]
The matrix $\Sh_k$ is obtained from the identity matrix by shifting entries to the right
so that the nonzero entries are placed along the NW-SE diagonal starting at $(1,k+1)$.
For an infinite matrix $A$, the `$k$-th trace' 
\[
\tr_k(A):= \tr(\Sh_k\cdot A)
\]
is the sum of the entries in the $k$-th subdiagonal, provided of course that this 
sum is defined, for example as a formal power series.

It is well-known that $\fs$ and $\fS$ are inverses of each other. We consider the 
following matrix, obtained as the product of the commutator of $\one_\Lbb$
and $\fs$ by $\one_\Lbb$:
\[
\one_\Lbb\cdot \fs \cdot \one_{\Lbb^{-1}}\cdot \fS \cdot \one_\Lbb
=\begin{pmatrix}
1 & 0 & 0 & 0 & \cdots \\[3pt]
1-\Lbb & \Lbb & 0 & 0 & \cdots \\[3pt]
1-3\Lbb+2\Lbb^2 & 3\Lbb-3\Lbb^2 & \Lbb^2 & 0& \cdots \\[3pt]
1-6\Lbb+11\Lbb^2-6\Lbb^3 & 7\Lbb-18\Lbb^2+11\Lbb^3 & 6\Lbb^2 - 6\Lbb^3 &
\Lbb^3 & \cdots \\[3pt]
\vdots & \vdots & \vdots & \vdots & \ddots
\end{pmatrix}\saf.
\]
With this notation, the following is a restatement of~Theorem~\ref{thm:M0nbarGC}.

\begin{theorem}\label{thm:Stma}
\[
[\ocaM_{0,n}] =(1-\Lbb)^{n-1} \cdot \tr_{n-2}\left(\one_\Lbb\cdot \fs\cdot 
\one_{\Lbb^{-1}}\cdot \fS\cdot \one_\Lbb\right)\saf.
\]
\end{theorem}

The trace appearing in this statement is well-defined as a series: since the $i^\mathrm{th}$ 
column of the matrix is a multiple of $\Lbb^{i-1}$, only finitely many entries on every diagonal 
contribute to the coefficient of each power of $\Lbb$.

\begin{example}\label{ex:tr4}
The entries in the fourth subdiagonal are
\begin{alignat*}{13}
& 1 && -10\, &&\Lbb && +\,\,\, 35 &&\Lbb^2 && -\,\,\,\,\,\,50 &&\Lbb^3 
&& +\,\,\,\,\,\,24 &&\Lbb^4 &&  && && && \\
&  && \phantom{\,+\,}\,31&&\Lbb && -225\,&&\Lbb^2 && +\,\,\,595 &&\Lbb^3 && 
-\,\,\,675 &&\Lbb^4 && +\,\,\,\,\,\,274 &&\Lbb^5 &&  && \\
& && && && \phantom{\,+\,}\, 301 &&\Lbb^2 && -1890\, &&\Lbb^3 && 
+4375\, &&\Lbb^4 && -\,\,\,4410 &&\Lbb^5 && +\,\,\,1624 &&\Lbb^6 \\
& && && && && && \phantom{\,+\,}\,1701 &&\Lbb^3 && 
-9800\, &&\Lbb^4 && +20930\, &&\Lbb^5 && -19600\, &&\Lbb^6 +\,\,\, 6769\,\Lbb^7 \\
& && && && && && && && \phantom{\,+\,}\,6951 &&\Lbb^4 && 
-37800\, &&\Lbb^5 && +76440\, &&\Lbb^6 -68040\, \Lbb^7 + 22449\, \Lbb^8
\end{alignat*}
etc., adding up to
\[
\tr_4\left(\one_\Lbb\cdot \fs\cdot \one_{\Lbb^{-1}}\cdot \fS\cdot \one_\Lbb\right)
=1+21\Lbb+111\Lbb^2+356\Lbb^3+875\Lbb^4+\cdots
\]
and
\[
(1-\Lbb)^5\cdot 
\tr_4\left(\one_\Lbb\cdot \fs\cdot \one_{\Lbb^{-1}}\cdot \fS\cdot \one_\Lbb\right)
= 1+16\Lbb+16\Lbb^2+\Lbb^3=[\ocaM_{0,6}]
\]
as it should.
\qede\end{example}

\begin{proof}[Proof of Theorem~\ref{thm:Stma}]
The $(a,b)$-entry of the matrix
$\one_\Lbb \cdot \fs \cdot \one_{\Lbb^{-1}} \cdot \fS \cdot \one_\Lbb$
is
\[
(\one_\Lbb \cdot \fs \cdot \one_{\Lbb^{-1}} \cdot \fS \cdot \one_\Lbb)_{(a,b)}
=\sum_c \Lbb^{a-1} s(a,c) \Lbb^{-(c-1)} S(c,b) \Lbb^{b-1}
=\sum_c s(a,c) S(c,b) \Lbb^{a+b-c-1}\saf.
\]
For $a=k+n-1$, $b=k+1$, this gives
\[
\sum_{j\ge 0} s(k+n-1,k+n-1-j)\, S(k+n-1-j,k+1)\, \Lbb^{k+j}\saf.
\]
By Theorem~\ref{thm:M0nbarGC},
\begin{align*}
[\ocaM_{0,n}]&=(1-\Lbb)^{n-1}\cdot 
\sum_{k\ge 0} (\one_\Lbb \cdot \fs \cdot \one_{\Lbb^{-1}} 
\cdot \fS \cdot \one_\Lbb)_{(k+n-1,k+1)} \\
&=(1-\Lbb)^{n-1}\cdot 
\tr_{n-2}(\one_\Lbb \cdot \fs \cdot \one_{\Lbb^{-1}} \cdot \fS \cdot \one_\Lbb)
\end{align*}
as stated.
\end{proof}

The product by $(1-\Lbb)^{n-1}$ can also be accounted for in terms of this matrix
calculus, giving
\[
[\ocaM_{0,n}] = \tr_{n-2}\left((1-\Lbb)\cdot \one_{1-\Lbb} \cdot \one_\Lbb\cdot \fs\cdot 
\one_{\Lbb^{-1}} \cdot \fS \cdot \one_\Lbb \cdot \one_{{(1-\Lbb)}^{-1}}\right)
\saf.
\]

%%%

\section{About the Euler characteristic of $\ocaM_{0,n}$}\label{sec:OtEcoM}

We will end with a few comments on the Euler characteristic $\chi(\ocaM_{0,n})$ of
$\ocaM_{0,n}$. The generating function
\[
\hchi(z):=1+z+\sum_{n\ge 3} \chi(\ocaM_{0,n})\frac{z^{n-1}}{(n-1)!}
\]
is the specialization of $\hM$ at $\Lbb=1$, but cannot be recovered from the expressions
obtained in Theorem~\ref{thm:main} because of convergence issues. However, 
specializing the differential equation~\eqref{eq:Mdifeq} for $\hM$ at $\Lbb=1$ gives
\[
\frac{d\hchi}{dz}= \frac{\hchi}{2+z-\hchi}
\]
with initial condition $\hchi(0)=1$ (cf.~\cite[(0.10)]{MR1363064}, with different notation),
whose solution is
\begin{equation}\label{eq:hchiLam}
\hchi(z)=\frac{z+2}{-W_{-1}(-(z+2)e^{-2})}=
1+z+\frac{z^2}{2!}+2\,\frac{z^3}{3!} + 7\,\frac{z^4}{4!}+34\, \frac{z^5}{5!}
+213\, \frac{z^6}{6!}+\cdots\saf.
\end{equation}
Here, $W_{-1}$ is the other real branch of the Lambert W-function, cf.~\cite{MR1414285}.
\begin{center}
\includegraphics[scale=.6]{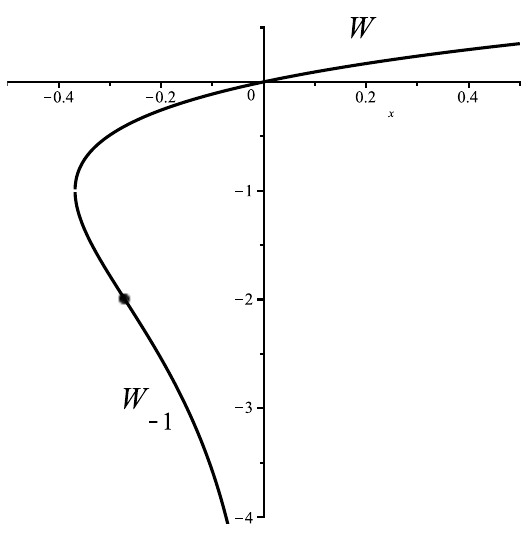}
\end{center}

We do not know how to extract an explicit expression for $\chi(\ocaM_{0,n})$ 
from~\eqref{eq:hchiLam}; but any of the formulas obtained for the Betti numbers in this
note yields such an expression. For example, Corollary~\ref{cor:bettisim} implies
\[
\chi(\ocaM_{0,n})
=\sum_{\ell=0}^{n-3}\sum_{j=0}^\ell \sum_{k=0}^{\ell-j} (-1)^{\ell-j-k}\binom{n-1}{\ell-j-k}
s(k+n-1,k+n-1-j)S(k+n-1-j,k+1)\saf.
\]
An alternative formulation may be obtained in terms of the shifted trace
\[
\tr_{n-2}\left(\one_\Lbb \cdot \fs\cdot \one_\Lbb^{-1}\cdot \fS\cdot \one_\Lbb\right)
=\sum_{j\ge 0} s(k+n-1,k+n-1-j)\, S(k+n-1-j,k+1)\, \Lbb^{k+j}
\]
appearing in~Theorem~\ref{thm:Stma}.

\begin{prop}\label{prop:Eucha}
\[
\tr_{n-2}\left(\one_\Lbb \cdot \fs\cdot \one_\Lbb^{-1}\cdot \fS\cdot \one_\Lbb\right)
=\sum_{k\ge 0}a_n(k) \Lbb^k
\]
where $a_n(k)$ is a polynomial with rational coefficients, degree $n-2$, and leading 
term $\frac{\chi(\ocaM_{0,n})}{(n-2)!}$.
\end{prop}

\begin{proof}
The class $[\ocaM_{0,n}]$ is a polynomial of degree $n-3$ in $\Lbb$, therefore there exist
integers $b_0,\dots, b_{n-3}$ such that
\[
[\ocaM_{0,n}]=\sum_{\ell=0}^{n-3} b_\ell(1-\Lbb)^\ell\saf,
\]
and with this notation $b_0=[\ocaM_{0,n}]|_{\Lbb=1}=\chi(\ocaM_{0,n})$.
By Theorem~\ref{thm:Stma},
\begin{align*}
\tr_{n-2}\left(\one_\Lbb \cdot \fs\cdot \one_\Lbb^{-1}\cdot \fS\cdot \one_\Lbb\right)
&=\frac{[\ocaM_{0,n}]}{(1-\Lbb)^{n-1}} 
=\sum_{\ell=0}^{n-3} \frac{b_\ell}{(1-\Lbb)^{n-\ell-1}} \\
&=\sum_{k\ge 0} \left(\sum_{\ell=0}^{n-3} b_\ell \binom{k+n-\ell-2}{n-\ell-2}\right) \Lbb^k\saf.
\end{align*}
It follows that the coefficient $a_n(k)$ of $\Lbb^k$ is a polynomial as stated, and
\[
a_n(k) = b_0\,\frac{k^{n-2}}{(n-2)!} + \text{lower order terms}\saf,
\]
concluding the proof.
\end{proof}

\begin{example}
In Example~\ref{ex:tr4} we noted
\[
\tr_4\left(\one_\Lbb\cdot \fs\cdot \one_{\Lbb^{-1}}\cdot \fS\cdot \one_\Lbb\right)
=1+21\Lbb+111\Lbb^2+356\Lbb^3+875\Lbb^4+\cdots
\]
Tracing the argument in the proof of Proposition~\ref{prop:Eucha} shows that these
coefficients are values of the polynomial
\[
a_6(k) = \frac{17}{12}\,k^4+\frac{17}3\,k^3+\frac{97}{12}\,k^2+\frac{29}6\,k+1\saf,
\]
and $\ocaM_{0,6}=34$.
\qede\end{example}

%%%

\appendix

%%%

\section{Combinatorial identities}\label{sec:app}

Here are collected the combinatorial statements used in this paper; 
proofs are included for completeness.

\begin{lemma}\label{lem:binid1}
For all integers $k,w\ge 1$,
\[
\sum_{\ell+m=k-1}
\binom{w(\ell+1)-1}\ell \cdot
\binom{w(m+1)-2}m 
\cdot \frac 1{m+1}
=\binom{w(k+1)-2}{k-1}\saf.
\]
\end{lemma}

\begin{proof}
We will use one form of the Lagrange series identity, \cite[\S4.5]{MR231725}:
\[
\frac{f(x)}{1-y \varphi'(x)}=\sum_{r\ge 0} \frac{y^ r}{ r!} \left[ \frac {d^ r}{dx^ r}
(f(x)\varphi^ r(x))\right]_{x=0}\saf,
\]
where $y=\frac x{\varphi(x)}$. Applying this identity with $f(x)=(x+1)^\alpha$,
$\varphi(x)=(x+1)^\beta$ gives
\begin{equation}\label{eq:Lagse}
\frac{(x+1)^{\alpha+1}}{x+1-\beta x}=\sum_{ r\ge 0} \binom{\alpha+\beta  r} r y^ r\saf.
\end{equation}
Setting $\alpha=w-1,\beta=w$ in~\eqref{eq:Lagse} gives
\[
A(x):=\sum_{\ell\ge 0} \binom{w(\ell+1)-1}\ell y^\ell = \frac{(x+1)^w}{x+1-w x}\saf.
\]
Setting
\[
B(x):=\sum_{m\ge 0} \binom{w(m+1)-2}m \cdot \frac {y^m}{m+1}
\]
and applying~\eqref{eq:Lagse} with $\alpha=w-2, \beta=w$ gives
\[
\frac{d}{dy}(By)=\sum_{m\ge 0} \binom{w(m+1)-2}m y^m= \frac{(x+1)^{w-1}}{x+1-w x}\saf,
\]
from which
\[
By=\int \frac{(x+1)^{w-1}}{x+1-w x} dy 
=\int \frac{(x+1)^{w-1}}{x+1-w x}\cdot \frac{dy}{dx} dx
=\int \frac 1{(x+1)^2} dx =-\frac 1{x+1} +C\saf.
\]
Setting $x=0$ determines $C=1$, and we get
\[
B(x)=(x+1)^{w-1}\saf.
\]
Another application of~\eqref{eq:Lagse}, with $\alpha=2w-2,\beta=w$, gives
\begin{align*}
\sum_{\ell\ge 0}
&\binom{w(\ell+1)-1}\ell y^\ell\cdot
\sum_{m\ge 0}
\binom{w(m+1)-2}m \cdot \frac {y^m}{m+1}\\
& = A(x)B(x)=\frac{(x+1)^{2w-1}}{1-x(w-1)} \\
&=\sum_{ r\ge 0} \binom{w( r+2)-2}{ r} y^ r\saf.
\end{align*}
Extracting the coefficient of $y^{k-1}$ verifies the stated identity.
\end{proof}

\begin{lemma}\label{lem:binid2}
For all positive integers $\alpha, \beta$, we have
\[
x^\alpha = \sum_{\ell\ge 0} \frac \alpha {\alpha+\ell\beta}\binom{\alpha+ \ell\beta}\ell
y^\ell\saf,
\]
where $y=(x-1)x^{-\beta}$.
\end{lemma}

\begin{proof}
We will use a second form of the Lagrange series identity, \cite[\S4.5]{MR231725}:
\[
f(x)=f(0)+\sum_{\ell=1}^\infty \frac{y^\ell}{\ell!} 
\left(\frac {d^{\ell-1}}{dx^{\ell-1}}(f'(x)\varphi^\ell(x))\right)|_{x=0}
\]
where $y=\frac x{\varphi(x)}$. Apply this identity with $f(x)=(x+1)^\alpha$, 
$\varphi(x)=(x+1)^\beta$, $y=x(x+1)^{-\beta}$. We have
\begin{align*}
\frac 1{\ell!} \frac {d^{\ell-1}}{dx^{\ell-1}}(\alpha (x+1)^{\alpha+\ell\beta-1})|_{x=0}
&=\frac \alpha{\ell!} (\alpha+\ell\beta-1)\cdots (\alpha_\ell\beta-(\ell-1)) 
(x+1)^{\alpha+\ell\beta-\ell}|_{x=0} \\
&=\frac \alpha{\alpha+\ell\beta}\binom{\alpha+\ell\beta}{\ell}
\end{align*}
and hence
\[
(x+1)^\alpha=\sum_{\ell\ge 0} \frac \alpha {\alpha +\ell\beta} 
\binom{\alpha+\ell\beta}{\ell}(x(x+1)^{-\beta})^\ell\saf.
\]
Substituting $x-1$ for $x$ gives the statement.
\end{proof}

\begin{lemma}\label{lem:St2}
For $N,r>0$, let $S(N,r)$ denote the Stirling number of the second kind (cf.~\eqref{eq:Stir2def}).
Then for all integers $a\ge 2$ there exists a polynomial $\sigma_a(x)$ of degree $a-2$ such
that
\[
\sum_{N\ge 0} S(N+a,N+1) x^N = \frac{\sigma_a(x)}{(1-x)^{2a-1}}\saf.
\]
For $a=1$, the same holds with $\sigma_a(x)=1$.
\end{lemma}

\begin{remark}
The coefficients of the polynomials $\sigma_a$ are the numbers denoted $B_{k,i}$ 
in~\cite{MR0462961}, where it is proved (\cite[Theorem~2.1]{MR0462961}) that they 
equal the number of ``Stirling permutations of the multiset $\{1,1,2,2,\dots, k,k\}$ with 
exactly $i$ descents.'' In particular, they are positive integers.
\qede\end{remark}

\begin{proof}
Since $S(N+1,N+1)=1$ and $S(N+2,N+1)=\binom{N+2}2$ for all $N\ge 0$, the statement is 
immediate for $a=1$ and $a=2$. For $a>2$, let
\[
\St_a(x):= \sum_{N\ge 0} S(N+a,N+1) x^N = \frac{\sigma_a(x)}{(1-x)^{2a-1}}
\]
and recall the basic recursive identity satisfied by Stirling numbers of the second kind:
\[
S(N+a,N+1)=S(N+a-1,N)+(N+1) S(N+a-1,N+1)\saf.
\]
This implies the relation
\[
\St_a(x) = \frac 1{1-x}\cdot \frac d{dx}(x \St_{a-1}(x))
\]
from which the statement follows easily. In fact, this identity implies the recursion
\[
\sigma_a(x)=(1-x)(\sigma_{a-1}(x)+x \sigma'_{a-1}(x))+(2a-3)x\sigma_{a-1}(x)\saf,
\]
verifying that $\sigma_a(x)$ is a polynomial of degree $a-2$, with leading coefficient $(a-1)!$,
if $\sigma_{a-1}(x)$ is a polynomial of degree $a-3$ with leading coefficient $(a-2)!$.
\end{proof}

%%%

\newcommand{\etalchar}[1]{$^{#1}$}

\end{document}